\documentclass[10pt,a4paper]{amsart}
\usepackage{amsmath,amsfonts,amssymb,amscd,amsthm,amsbsy,bbm, epsf,calc,enumerate,color,graphicx,verbatim}
\usepackage[noadjust]{cite}
\usepackage{color}
\usepackage{datetime}
\usepackage{hyperref}
\usepackage[usenames,dvipsnames]{xcolor}
\usepackage{endnotes}
\usepackage[disable]{todonotes}

\usepackage{tikz}
\usepackage{caption}
\usepackage{subcaption}

\def\XXint#1#2#3{{\setbox0=\hbox{$#1{#2#3}{\int}$ }
\vcenter{\hbox{$#2#3$ }}\kern-.6\wd0}}

\newcounter{hours}\newcounter{minutes}

\theoremstyle{theorem}

\newtheorem{thm}{Theorem}[section]

\newtheorem{lem}[thm]{Lemma}

\newtheorem{cor}[thm]{Corollary}
\newtheorem{prop}[thm]{Proposition}

\newtheorem{THM}{Theorem}

\theoremstyle{definition}

\newtheorem{DEF}[thm]{Definition}
\newtheorem{hyp}{Assumption}

\theoremstyle{remark}                  

\newtheorem{rem}[thm]{Remark}

\makeatletter
\theoremstyle{theorem}
\newtheorem*{rep@theorem}{\rep@title}
\newcommand{\newreptheorem}[2]{%
	\newenvironment{rep#1}[1]{%
		\def\rep@title{#2 \ref{##1}}%
		\begin{rep@theorem}}%
		{\end{rep@theorem}}}
\makeatother

\newreptheorem{thm}{Theorem}
\newreptheorem{lem}{Lemma}
\newreptheorem{cor}{Corollary}
\newreptheorem{prop}{Proposition}

\makeatletter
\theoremstyle{definition}
\newtheorem*{REP@theorem}{\rep@title}
\newcommand{\newREPtheorem}[2]{%
	\newenvironment{REP#1}[1]{%
		\def\rep@title{#2 \ref{##1}}%
		\begin{REP@theorem}}%
		{\end{REP@theorem}}}
\makeatother

\newREPtheorem{DEF}{Definition}

\def\N{{\mathcal N}}

\def\R{{\mathbb R}}

\def\e{\varepsilon}

\def\d{\textnormal{d}}

\DeclareMathOperator*{\argmin}{\arg\!\min}
\def\limsup{\mathop{\lim\,\sup}\limits}%
\def\argmin{\mathop{\arg\,\min}\limits}%

\numberwithin{equation}{section}

\usepackage[utf8]{inputenc}
\usepackage{color}

\usepackage{amsmath}
\usepackage{amssymb}
\usepackage{amsthm}
\usepackage{fullpage}
\usepackage{enumerate}
\usepackage{dsfont}
\usepackage{stmaryrd}
\usepackage{graphicx}
\usepackage{mathrsfs}

\newcommand{\sd}{\textnormal{sd}}
\newcommand{\oQ}{\overline{Q}}
\newcommand{\rr}{\rho\textit{-reflection}}

\begin{document}

\title{On Mean Curvature Flow with forcing}
\author{Inwon Kim}
\address{Department of Mathematics, UCLA, Los Angeles, USA}
\thanks{Inwon Kim was partially supported by NSF DMS-1300445.}
\email{ikim@math.ucla.edu}

\author{Dohyun Kwon}
\address{Department of Mathematics, UCLA, Los Angeles, USA}
\email{dhkwon@ucla.edu}
\thanks{Dohyun Kwon was partially supported by National Institute for Mathematical Sciences(NIMS) funded by the Ministry of Science, ICT \& Future Planning(B21501)}
\date{}

\subjclass[2010]{}


\begin{abstract}
This paper investigates geometric properties and  well-posedness of a mean curvature flow with volume-dependent forcing. With the class of forcing which bounds the volume of the evolving set away from zero and infinity, we show that a strong version of star-shapedness is preserved over time. More precisely, it is shown  that the flow preserves the {\it $\rho$-reflection property}, which corresponds to a quantitative Lipschitz property of the set with respect to the nearest ball. Based on this property we show that the problem is well-posed and its solutions starting with $\rho$-reflection property become instantly smooth.  Lastly, for a model problem, we will discuss the flow's exponential convergence to the unique equilibrium in Hausdorff topology. For the analysis, we adopt the approach developed in \cite{Feldman:2014hb} to combine viscosity solutions approach and variational method.  The main challenge lies in the lack of comparison principle, which accompanies forcing terms that penalize small volume.
\end{abstract}

\maketitle

\section{Introduction}

In this paper we consider the sets $(\Omega_t)_{t>0}$ in $\R^n$ moving by the motion law 
\begin{equation}\label{main}
V = -H + \lambda[|\Omega_t|] \quad \hbox{ on } \partial\Omega_t.
\end{equation}
Here $V=V(x,t)$ and $H=H(x,t)$ respectively denote the outward normal velocity and the mean curvature of $\partial\Omega_t$ at $x\in\partial \Omega_t$, where $H$ is set to be positive if $\Omega_t$ is convex at the point.  The volume-dependent forcing $\lambda:\R^+ \to \R$ will be assumed to be locally Lipschitz with growth conditions (see Assumption A below). 

\medskip

We are interested in the global-time description of the flow, including its well-posedness. In general, due to the low-dimensional nature of the interface, finite-time topological singularities are expected even for interfaces starting out with smooth shapes. 
On the other hand \eqref{main} is a parabolic flow, and thus parabolic regularity theory applies once we know that the evolving boundary $\partial\Omega_t$ is locally a graph. Thus our first goal is to establish an a priori graph property of $\partial\Omega_t$ by studying the geometry of the evolution. 

\medskip

The forcing $\lambda$ we consider in this article keeps the volume of $\Omega_t$ bounded away from zero and infinity. With such choices of forcing we will show that a strong version of star-shapedness property holds for $\Omega_t$ at all $t>0$ if initially true, assuming the existence of the flow. Let us remark that this geometric result does not extend to the classical mean curvature flow where $\lambda = 0$.  With zero forcing and with star-shaped initial set, solutions of \eqref{main} have been shown to hold certain semi-convexity estimates by Smoczyk \cite{Smo1} and Lin \cite{Lin}. While these estimates allow classification of possible singularities for the flow in terms of blow-up limits,  it remains open whether the initially star-shaped flow stays star-shaped beyond the initial time even with zero forcing.

\medskip

While the zero forcing case appears to be out of reach at the moment, in the subsequent work \cite{KK19} we show that our analysis applies to the volume-preserving mean curvature flow, using approximate forcing $\lambda_{\delta} =\frac{1}{\delta}( |\Omega_0| - |\Omega|)$ with small $\delta>0$ which satisfies our assumptions. One of the main focus in \cite{KK19} will be on obtaining a sufficiently strong convergence of this approximation to preserve the flow's geometric properties.

\medskip

With the a priori geometric property of the flow, we next discuss existence and uniqueness of the flow \eqref{main} based on its variational structure.  A formal calculation yields the energy inequality 
\begin{align}\label{energy}
\frac{d}{dt} J(t) = -\int_{\partial\Omega_t} V^2 dS,
\end{align}
where $J(t) = \mathrm{Per}(\Omega_t) - \Lambda(|\Omega_t|)$ with $\Lambda$ the anti-derivative of $\lambda$ and $V$  as given in \eqref{main}. From \eqref{energy} one expects $\Omega_t$ to flow toward a stationary point of the energy as time grows. We will make this observation rigorous by generating a discrete-time approximation (or ``minimizing movement") that satisfies the energy dissipation.  The aformentioned a priori geometric property enables the uniform convergence of its discrete time approximations, to guarantee that in the continuum limit we recover a smooth solution. For long-time behavior, to illustrate our ideas, we show an exponential convergence of the sets for the following specific flow 
\begin{align}\label{main1}
V= -H + |\Omega_t|^{\gamma}, \hbox{ where } \gamma < -1/n
\end{align}
which has a simple energy structure with the unique stable equilibrium. 

\medskip

While the variational approach yields the minimizing movements approximation as well as the asymptotic analysis of the flow, viscosity solutions are more suited for  geometric arguments.  To take advantage of both approaches we will show that the variational flow is, in a sense, a viscosity solution of \eqref{main}. This idea of combining the two approaches were previously used for the mean curvature flow  in \cite{Chambolle:2004ds}, but in our problem the standard maximum principle does not apply for  \eqref{main}, and thus the notion of viscosity solutions needs to be modified from the standard one. Indeed our main novelty in the analysis is to combine these two approaches to address geometric motions which do not satisfy a comparison principle but still is of parabolic nature.  For free boundary problems this combination has been introduced in \cite{Feldman:2014hb}, where the presence of bulk pressure plays a crucial role in the analysis.
 
\medskip

To state the main results, let us begin with discussing the assumptions on the forcing.
\begin{hyp}
\label{hyp_a}
$\lambda : \R^+ \rightarrow \R$ is locally Lipschitz continuous and satisfies $\limsup\limits_{R\rightarrow \infty}\frac{\lambda[|B_R|]}{R} < \infty$. In addition, there exists $\rho>0$ such that $\lambda[|\Omega|] > \frac{n-1}{\rho}$ for  all $\Omega \subset \overline{B}_{5\rho}$.
\end{hyp}

The first part of the assumption is necessary to show that the evolution is unique and the set does not spread to $\R^n$ in finite time. The second part puts a sufficient penalty on shrinkage of the evolution, and is used in showing that the evolution always contains a small ball $B_{\rho}(0)$ if initially so (Lemma~\ref{Brho}). Both  $\lambda$ given in \eqref{main1} and $\lambda$ given by a large multiple of $C- |\Omega|$ satisfy Assumption A.

\medskip

With the parameter $\rho$ given from above assumption, we also assume the following on the initial data.
\begin{hyp}
\label{hyp_b}
$\Omega_0$ has $\rr$ (see Definition~\ref{rho-ref}).
\end{hyp}
The $\rho$-reflection property should be interpreted  as a quantitative smallness requirement
on the Lipschitz norm distance between $\Omega_0$ and the nearest ball (see Lemma~\ref{star}).

\medskip

We are now ready to summarize our results. We adopt Definition~\ref{def-sol} as the notion of solutions for \eqref{main}. Our first result states the preservation of the $\rho$-reflection property. The proof is based on the reflection maximum principle as well as various barrier arguments based on Assumption A.

\begin{THM} {\rm {[A priori geometric properties, Theorem~\ref{pre-star}]}}
Suppose that $\Omega_0$ has $\rr$ and there is a solution $(\Omega_t)_{t>0}$ of \eqref{main}. Then, $\Omega_t$ has $\rr$ at all times $t>0$.  In particular there exists $r_1= r_1(\rho)>0$ such that $\Omega_t$ is \textit{star-shaped with respect to a ball} $B_{r_1}(0)$ for all $t>0$.
\end{THM}

From above result and the volume bound it follows that $\Omega_t$ has locally Lipschitz boundary which is uniform in time. This fact endows sufficient compactness for the evolution that makes it possible for the discrete-time variational scheme to approximate the flow, in particular establishing the global existence results.
 
\begin{THM} {{\rm [Well-posedness, Theorem~\ref{coro}]}}
Suppose that $\Omega_0$ has $\rr$. Then, there exists a unique solution $(\Omega_t)_{t>0}$ of \eqref{main} that is bounded and has smooth boundary for every $t>0$. $(\Omega_t)_{t>0}$ can be approximated locally uniformly by minimizing movements with constraints.
\end{THM}
	
Lastly we discuss asymptotic convergence of the model flow \eqref{main1} with exponential convergence rate.
			
\begin{THM}{{\rm [Long-time behavior, Theorem \ref{exponential}]}}
Suppose that $\Omega_0$ has $\rr$.
For \eqref{main1}, there exists $T>0$ such that $\partial\Omega_t$ is uniformly $C^{1,1}$ for $t\geq T$. Moreover $\Omega_t$ converges to a ball of radius $r_\infty$ exponentially fast as $t\to\infty$, in Hausdorff distance.
\end{THM}

We expect that the above results can be extended to flows with more general forcing, such as the volume-preserving mean curvature flow. The main challenge there is in the lack of a priori regularity of $\lambda$. For further discussion we refer to \cite{KK19}, where Theorem 1 and 2 are shown for the volume preserving mean curvature flow. 

\bigskip

{\bf Literature}

\medskip

The viscosity solutions approach for \eqref{main} with fixed $\lambda$ was introduced by Evans and Spruck (\cite{Evans:1999tb}, \cite{Anonymous:JTDBs7kP}, \cite{Anonymous:UpA3UBgR},\cite{Evans:1995iw}), and by Chen, Giga and Goto \cite{Chen:1991u} in more general context. The minimizing movements were first introduced for 
\begin{align}\label{vanilla}
V= -H
\end{align} 
in \cite{Almgren:1993gw} and \cite{LS95}.  When the flow satisfies comparison principle, coincidence between flat flows and viscosity solutions has been established by Chambolle et al (\cite{Chambolle:2004ds}, \cite{CMP15}).

\medskip

Local regularity results are available for \eqref{vanilla} with the aforementioned unit density hypothesis, using the notion of varifold solutions from geometric measure theory (see Brakke \cite{Brakke}). Escher and Simonett \cite{Escher:1998} show that if the initial surface is sufficiently close to a sphere in $C^{1,\alpha}$ sense, then the volume preserving mean curvature flow converges to the sphere exponentially fast.

\medskip

Next we recall results addressing geometric properties of the flow.  In \cite{Huisken:69hmQ_XX}, Huisken showed that initially strictly convex surfaces evolving by \eqref{vanilla} stay convex and shrink to a point in finite time. Parallel results are shown in \cite{Hui87} for the volume-preserving mean curvature flow, where it is shown that convex surfaces converge into spheres.
These convergence results are extended respectively to anisotropic mean curvature flow by Andrews \cite{Anonymous:5CWDAe2w}, and to its volume-preserving version by Bellettini et al. in \cite{BCCN09}. 

\medskip

For non-convex sets, most available results investigates singularities of the flow \eqref{vanilla} for spatial dimensions larger than two (\cite{AAG1995}, \cite{HS1}, \cite{HS2}, \cite{Smo1}, \cite{Lin}).  For \eqref{vanilla} in space dimension two, Angenent proves in  \cite{A1}-\cite{A2} that the number of intersections of a pair of curves does not increase over the evolution, and in particular simply connected domains stay so until they shrink to a point.  Note that, for flows with forcing, topology of an evolving set may not be preserved even in two dimensions (see Figure ~\ref{fig_non_star}(A) and (B) for possible scenarios).

\begin{figure}[h]
	\centering
	\begin{subfigure}[t]{0.3\textwidth}
		\includegraphics[width=\textwidth]{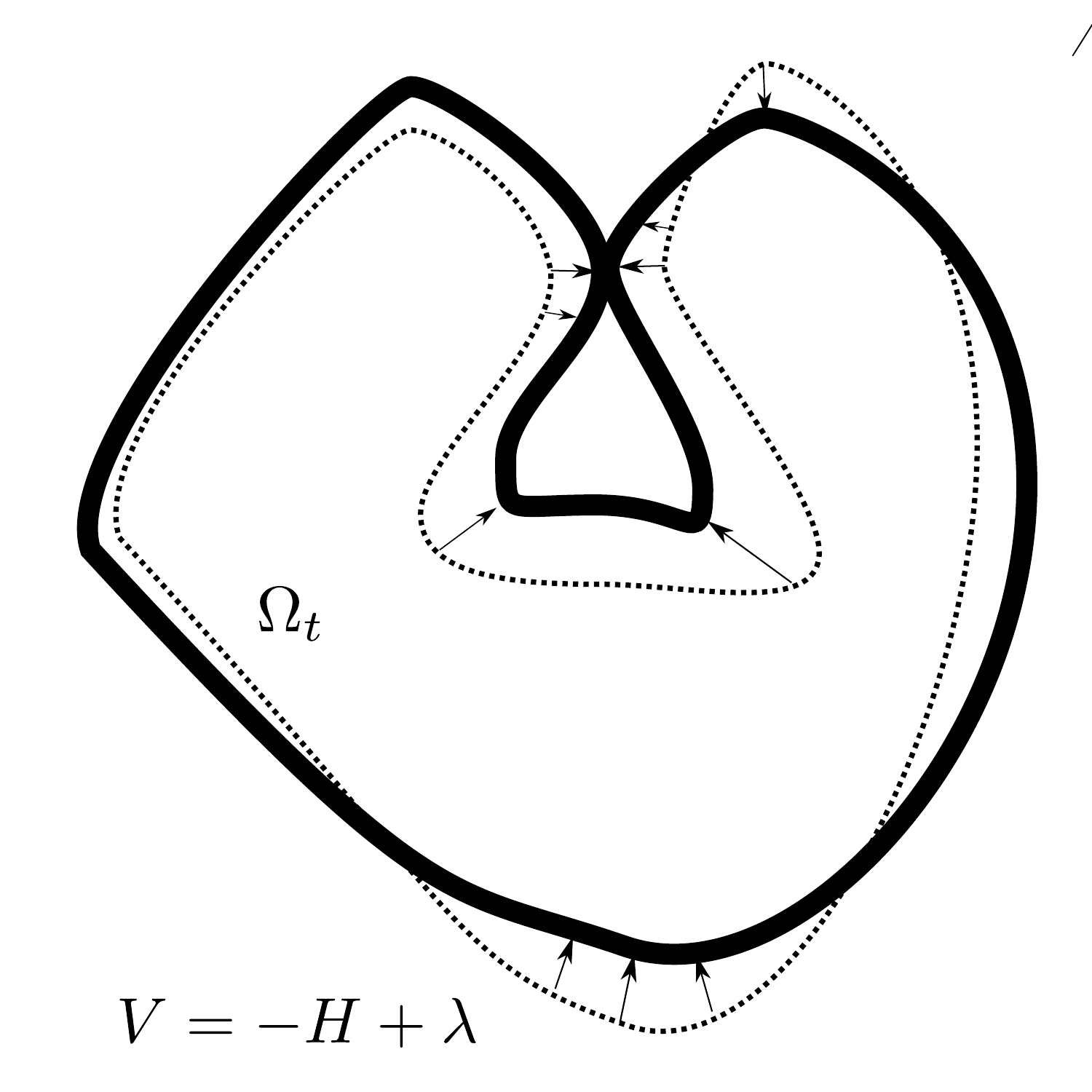}
		\caption{$\lambda>0$}
	\end{subfigure}
	\hspace{1cm}
	\begin{subfigure}[t]{0.3\textwidth}
		\includegraphics[width=\textwidth]{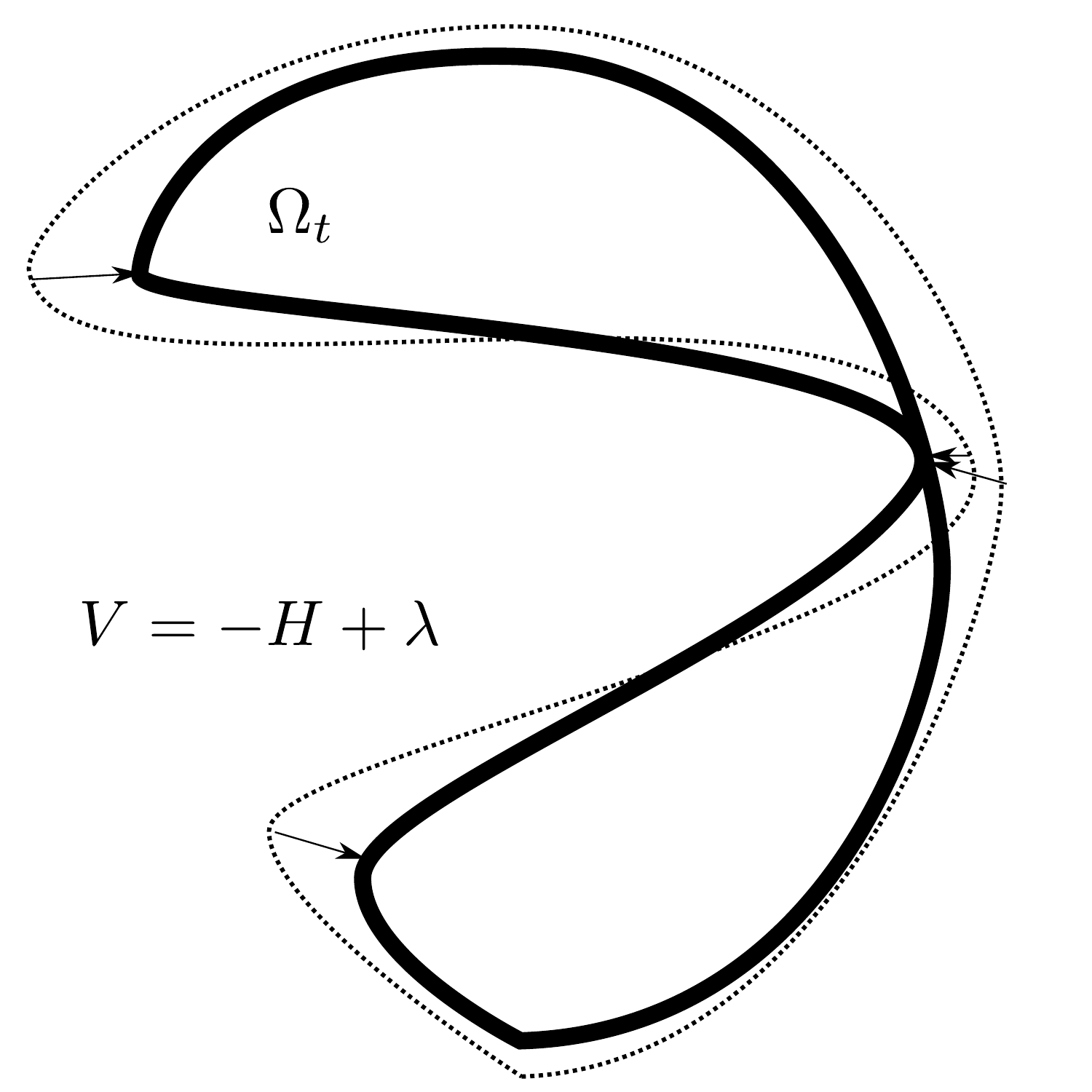}
		\caption{$\lambda< 0$}
	\end{subfigure}
	\caption{Singularities in Two Dimensions}
	\label{fig_non_star}
\end{figure}

\medskip

{\bf Outline of the paper }

In section ~\ref{sec_vis}  we give a definition on the notion of ``viscosity solutions" for \eqref{main} in terms of its level-set formulation.  To do so we first discuss the mean curvature flows with a priori fixed forcing,
\begin{align}\label{fixed}
V =  -H+ \eta(t).
\end{align}
Our solution $\Omega_t$  of \eqref{main} is then defined as the viscosity solution of \eqref{fixed} where $\eta(t)$ coincides with $\lambda[|\Omega_t|]$.

\medskip

In section \ref{sec_pre}  we show that \eqref{main} preserves the $\rho$-reflection property. As in \cite{Feldman:2014hb} our arguments are based on  reflection comparisons. More precisely,  for given $\nu,x_0 \in\R^n$  define $\Pi^+_{\nu,x_0}:= \{x+x_0: x\cdot\nu \geq 0\}$ and $\Pi_{\nu}:= \partial \Pi^+_{\nu,x_0}$. Since the normal velocity law  \eqref{main} is preserved with respect to spatial reflections, comparison principle applies in the region $\Pi^+_{\nu} \times [0,\infty)$ to $\Omega_t$ and $\Omega^{\nu,x_0}_t$, the reflected version of $\Omega_t$ with respect to $\Pi_{\nu}$.  It follows that if 
\begin{align}\label{order_reflection}
\Omega^{\nu,x_0}_0 \subset \Omega_0\hbox{ in }\Pi^+_{\nu},
\end{align}
then such property is preserved for later times. We will show that this property and Assumption B imply that $\partial\Omega_t$ is locally Lipschitz, as long as $\Omega_t$ contains a small neighborhood of the origin. Recall that \eqref{main} does not satisfy classical comparison principle. This is why we resort exclusively to this particular type of comparison arguments. 

\medskip

Section~\ref{sec_uni} yields uniqueness of solutions for \eqref{main}. The proof is based on small-time uniqueness for star-shaped solutions of \eqref{fixed}, and the Lipschitz continuity of $\lambda$ given by Assumption~\ref{hyp_a}.

\medskip

In section~\ref{sec_ene}, based on the discrete-time minimizing movement, we generate a {\it flat flow} of \eqref{main}  characterized as the continuum gradient flow of the energy functional $J(E)$ given in \eqref{energy}. Let us mention that, due to the lack of comparison principle, we need strong convergence of the discrete flow to characterize the continuum  limit. To this end we impose geometric constraints to the minimizing movement to generate sufficient compactness on the discrete flow: see Definition~\ref{dis-flow} and $\eqref{admissible}$.

\medskip

Section~\ref{sec_exi} discusses coincidence of the two notions of solutions. Based on Proposition \ref{cor_bar}, we show in Theorem \ref{coro} that the flat flow is the unique viscosity solution of \eqref{fixed} with $\eta(t) = \lambda[|\Omega_t|]$.

\medskip

Finally in section \ref{sec_reg} we address the  large-time behavior for \eqref{main1}. The gradient flow structure of \eqref{main} yields that $\Omega_t$  converges to a ball   in Hausdorff distance (Theorem~\ref{first}). Furthermore  we show that the convergence is in almost-$C^{1,1}$ sense (Lemma~\ref{holder}). Such regularity result invokes the center manifold approach taken in \cite{Escher:1998} to yield the exponential convergence of $\Omega_t$ to a unique ball of radius $r_{\infty}$.

\section{Viscosity Solutions}
\label{sec_vis}

Equation \eqref{main} can be formulated in terms of level sets, which allows us to introduce the notion of viscosity solutions for the flow. More precisely, for $Q:=\R^n \times (0, \infty)$ and $u : \oQ \rightarrow \R$,  let us define
\begin{align*}
\Omega_t = \Omega_t(u) := \{ x \in \R^n \mid u(x,t)>0 \} \hbox{ for } t \geq 0
\end{align*} 
and consider the following corresponding PDE of mean curvature flows with forcing:
\begin{align}
\tag{MF}
\label{main eq1}
\dfrac{u_t}{|Du|}(x,t) =  \nabla \cdot \left(\dfrac{Du}{|Du|} \right)(x,t) + \lambda[|\Omega_t(u)|] \hbox{ for } (x,t) \in Q.
\end{align} 

In this section, we introduce a weak notion of solutions for \eqref{main eq1}. To this end we first introduce $\eta(\cdot) : [0, +\infty) \rightarrow \R$ as an \textit{a priori} known continuous function of time $t$, and consider \begin{align}\label{main eq2}
\dfrac{u_t}{|Du|}(x,t) =  \nabla \cdot \left(\dfrac{Du}{|Du|}\right)(x,t) + \eta(t)\end{align}
with initial data
\begin{align}\label{initial}
u(x,0) = u_0(x) := \chi_{\Omega_0} - \chi_{\Omega_0^C} \hbox{ for } x \in \R^n.
\end{align}

\medskip

We begin by a list of definitions. 

\medskip

\noindent$\bullet$ For a open set $U \subset \R^n$, we define the parabolic cylinder $U_T := U \times (0,T]$ and the parabolic boundary of $U_T$, $\partial_p U_T := \overline{U}_T - U_T$. We define $Q_T:= \R^n \times (0,T]$.

\noindent$\bullet$ We denote $\mathcal{S}^{n\times n}$ as the space of $n \times n$ real symmetric matrices. 

\noindent$\bullet$  For a function $h:\oQ\to \R$ we denote its positive set by $\Omega_t(h):=\{x \in \R^n: h(x,t)>0\}$ for $t \geq 0$.

\noindent $\bullet$ For  $f : U \subset \R^d \rightarrow \R$ and $d \in \mathbb{N}$, we denote the lower and upper semi-continuous envelope of $f$, $f_*$ and $f^*: \overline{U} \to \R$  by 
\begin{align}
f_*(x) := \lim_{\e \downarrow 0} \inf_{\substack{|x-y| < \e,\\ y \in U}} f(y) \hbox{ and } f^*(x) := \lim_{\e \downarrow 0} \sup_{\substack{|x-y| < \e,\\ y \in U}} f(y)
\end{align}
where $\overline{U}$ denotes the closure of $U$ in $\R^d$.

\noindent$\bullet$
For two evolving sets $\Omega_t$ and $\tilde{\Omega}_t$, we say that its boundary $\partial \Omega_t$ touches $\partial\tilde{\Omega}_t$ from inside (or outside, respectively) at $(x_0,t_0) \in Q$ if there exists a neighborhood $\N$ of $(x_0, t_0)$ such that $\Omega_t \cap \N \subset \tilde{\Omega}_t \cap \N$ (or $\tilde{\Omega}_t \cap \N \subset \Omega_t \cap \N$, respectively) in $[0,t_0]$, $(\partial \Omega_t \cap \partial \tilde{\Omega}_t) \cap \N = \emptyset$ in $[0,t_0)$ and $x_0 \in \partial \Omega_{t_0} \cap \partial \tilde{\Omega}_{t_0}$.

\noindent$\bullet$
For a set $U$ in $\R^d$ and $d \in \mathbb{N}$, we denote the signed distance function by
\begin{align}
\label{eqn:sd}
\sd(x,U) := \d(x,U) - \d(x,U^C).
\end{align}
We use the convention that $\sd(x,U) := \infty$ if $U$ is empty and $\sd(x,U) := -\infty$ if $U^c$ is empty.

\medskip

For later purposes we also consider the {\it restricted} flow
\begin{align}\label{main_eq22}
\dfrac{u_t}{|Du|}(x,t) =  \max \left\{ \nabla \cdot \left(\dfrac{Du}{|Du|}\right)(x,t) + \eta(t), - M \right\} \hbox{ for } (x,t) \in Q.
\end{align}
Here, $M$ is a given positive constant, which will be fixed large enough later on.

\medskip

Now we recall  the definition viscosity solutions for equations \eqref{main eq2} and \eqref{main_eq22}. Let us denote $A:=( \R^n \setminus \{0\})  \times \mathcal{S}^{n \times n}\times [0,\infty)$ and define $F : A  \rightarrow \R$ by
$$F(p,X,t) := \text{trace}\left( \left(I - \dfrac{p}{|p|} \otimes \dfrac{p}{|p|}\right)X \right) + \eta(t) |p|.$$
Then, the equation  \eqref{main eq2} can be rewritten in the form of 
$$u_t = F(Du, D^2u,t).$$ 
Since the set $A$ is dense in $\R^n  \times \mathcal{S}^{n \times n}\times [0,\infty)$, the envelopes $F_*$ and $F^*$ are well-defined in $\R^n  \times \mathcal{S}^{n \times n}\times [0,\infty)$ with value in $\R \cup \{ \pm \infty \}$. 

\medskip

Recall a test function from \cite[Definition 3.2]{ImbSil13}. We say that a function $\phi : Q \to \R$ is \textit{a test function} on $Q$ is  if $\phi$ is $C^2$ with respect to $x$ and $C^1$ with respect to $t$.

\begin{DEF} \cite[Definition 2.1]{Chen:1991u}, \cite[Definition 6.1]{Barles:2013gx} \label{def-vis}
\begin{itemize}
\item[(a)] A function $u: \oQ \to \R$ is \textit{a viscosity subsolution} of \eqref{main eq2} if  $u^*< +\infty$ and for any test function $\phi$ on $Q$ that touches $u^* $ from above at $(x_0,t_0)$ we have
\begin{align*}
\phi_t(x_0,t_0) \leq  F^*(D\phi(x_0,t_0), D^2\phi(x_0,t_0), t_0).
\end{align*}

\item[(b)] A function $u: \oQ\to \R$ is \textit{a viscosity supersolution} of \eqref{main eq2} if $u_*>-\infty$ and for any test function $\phi$ on $Q$ that touches $u_* $ from below at $(x_0,t_0)$ we have
\begin{align*}
\phi_t(x_0,t_0) \geq  F_*(D\phi(x_0,t_0), D^2\phi(x_0,t_0),t_0).
\end{align*}
\item[(c)] A function $u: \oQ\to \R$ is \textit{a viscosity solution} of \eqref{main eq2} with initial data $u_0:\R^n \to \R$
if $u^*$ is \textit{a viscosity subsolution} and $u_*$ is \textit{a viscosity supersolution}, and if $u^*= {(u_0)}^*$ and $u_*= {(u_0)}_*$ at $t=0$. 
\end{itemize}
\end{DEF}

Parallel definitions can be made for viscosity sub- and supersolutions of \eqref{main_eq22}. 

\begin{thm}
\label{comparison}
$\,$
\begin{enumerate}
\item
\cite[Theorem 2.1]{GGIS}
Let $T>0$ and $U$ be a domain in $\R^n$, not necessarily bounded. Let $u$ and $v$ be a bounded subsolution and supersolution, respectively, of \eqref{main eq2} (or \eqref{main_eq22}).
If $u^* \leq v_*$ on $\partial_p U_T$, then we have $u^* \leq v_* \hbox{ on } U_T$.
\item
\cite[Theorem 1.1]{Barles:1993gaba}
For a given bounded domain $\Omega_0\subset \R^n$ and uniformly continuous initial data $u_0 : \R^n \to \R$ such that $\Omega_0 = \{ x \in \R^n : u_0(x) = 0\}$, there exists a unique viscosity solution $u$ of \eqref{main eq2} (or \eqref{main_eq22}), which is uniformly continuous in $\oQ$.
\item 
\cite[Theorems 1.1]{Barles:1993gaba} Let $u$ and $v$ be a uniformly continuous subsolution and supersolution, respectively, of \eqref{main eq2} (or \eqref{main_eq22}) in $\oQ$. If $u(\cdot, 0) \leq v(\cdot, 0)$ in $\R^n$, then we have $u \leq v$ in $\oQ$.
\end{enumerate}
\end{thm}

The following lemma is a consequence of the stability properties of viscosity solutions: see for instance Lemma 6.1 in \cite{Crandall:1992kn}.

\begin{lem}
$\,$	\label{lem_sta}
\begin{itemize}
\item[(a)]
For $n\in\mathbb{N}$, let $u_n:=\chi_{\Omega^n_t} -\chi_{(\Omega^n_t)^C}$ be a viscosity solution of \eqref{main eq2}(or \eqref{main_eq22}) in $Q$. If $\partial\Omega_t^n$ converges to $\partial\Omega_t$ as $n\to\infty$ in Hausdorff distance, uniformly for all $t>0$,  then $u:= \chi_{\Omega_t}-\chi_{\Omega_t^C}$ is a viscosity solution of \eqref{main eq2} (or \eqref{main_eq22}).
\medskip
\item[(b)]
For $n\in\mathbb{N}$, let $u_n:=\chi_{\Omega^n_t} -\chi_{(\Omega^n_t)^C}$ be a viscosity solution of \eqref{main_eq22} in $Q$ with $M=M_n$ where $M_n \rightarrow \infty$ as $n \rightarrow \infty$ and with initial data $u_0$. 
If $\partial\Omega_t^n$ uniformly converges to $\partial\Omega_t$ in Hausdorff distance, uniformly for all $t>0$, then $u$ is a viscosity solution of \eqref{main eq2}.
\end{itemize}
\end{lem}

Note that \eqref{main_eq22} as well as \eqref{main eq2} are {\it geometric}, that is $F$ satisfies the scaling invariance
\begin{align}
F(ap, aX + b p \otimes p,t ) = aF(p,X,t)
\end{align} 
for $a>0$, $b\in \R$, $p \in \R^n$, $X \in  \mathcal{S}^{n \times n}$ and $t>0$.
So, \eqref{main eq2} and \eqref{main_eq22} have the following invariance of geometric equations.

\begin{thm}
\cite[Theorem 4.2.1]{Gig06}
\label{thm-invar}	
Let $u$ and $v$ be a subsolution and supersolution, respectively, of \eqref{main eq2} (or \eqref{main_eq22}). If $\phi : \R \rightarrow \R$ is upper semicontinous and nondecreasing, then the composite function $\phi \circ u$ is also a subsolution. Similarly, if $\phi : \R \rightarrow \R$ is lower semicontinous and nondecreasing, then $\phi \circ v$ is also a supersolution.
\end{thm}

Let  $v$ be a continuous viscosity solution of \eqref{main eq2} with uniformly continuous initial data $u_0 : \R^n \to \R$ such that $\Omega_0 = \{ x \in \R^n : u_0(x) = 0\}$. Based on the invariance in Theorem \ref{thm-invar} and the stability of viscosity solutions in \cite[Lemma 6.1]{Crandall:1992kn}, we obtain a discontinuous viscosity solution $u$ of \eqref{main eq2} and \eqref{initial} given by
\begin{align}
\label{eq421}
u(x,t) = \chi_{\Omega_t(u)} - \chi_{(\Omega_t(u))^C} \hbox{ and } \Omega_t(u) = \Omega_t(v) \hbox{ for all } t \geq 0
\end{align}
(See \cite[Theorem 2.1]{Barles:1993gaba}). Note that $\Omega_t(u)$ satisfies \eqref{fixed} if $\partial\Omega_t$ is $C^2$. We will thus consider the set $\Omega_t$ obtained from the above viscosity solutions formulation as a weak notion of sets evolving by \eqref{fixed}.

\begin{rem}
\label{rem_uni}
Note that in Theorem~\ref{comparison}(1), we need $u^* \leq v_*$ at the initial time, so this theorem does not yield the uniqueness for discontinuous solutions. Indeed solutions of the form \eqref{eq421} may be non-unique due to the ``fattening" of the zero level set, see the discussion in \cite{Barles:1993gaba}, \cite{Evans1992}, \cite{Gig06} and \cite{soner1993} for the flow \eqref{vanilla}. We will show in section ~\ref{sec_uni} that our solutions are unique under the geometric constraint on the initial data.
\end{rem}

\begin{DEF} 
\label{def-sol}
A function $u: \oQ \to \R$ is \textit{a viscosity subsolution (supersolution)} of \eqref{main eq1} and \eqref{initial} if $u$ is a viscosity subsolution (supersolution) of \eqref{main eq2} and \eqref{initial} with continuous and bounded $\eta(t) = \lambda[{|\Omega_t(u)|}]$. A function $u$ is \textit{a viscosity solution} of \eqref{main eq1} and \eqref{initial} if $u$ is a viscosity solution of \eqref{main eq2} and \eqref{initial}.
\end{DEF}

\begin{rem} 
\label{rem_nonlocal}
For \eqref{main eq1} and \eqref{initial}, the comparison principle fails, and thus viscosity solutions theory cannot be directly applied. Indeed the well-posedness of \eqref{main eq1} and \eqref{initial} will be established later in section~\ref{sec_ene}.
\end{rem}

Next we introduce a regularization that is often used in free boundary problems (see e.g. \cite{Caffarelli:2005tk} and Lemma 3.1 in \cite{Kim:2002ta}). This is useful in our geometric analysis in sections~\ref{sec_uni} and \ref{sec_exi}.

\begin{lem}
\label{sup convolution_gen}
Consider a continuous function $l : [0,\infty) \rightarrow \R$ with  $L(t):= \int_0^t l(s) ds \leq A$ in $[0,T]$.
Let u be \textit{a viscosity supersolution} of \eqref{main eq2}. Then, the function
\begin{align*}
\tilde{u}(x,t) := \inf_{y \in \overline{B}_{A-L(t)}(x)} u(y,t),
\end{align*}
is \textit{a viscosity supersolution} of 
\begin{align}
\label{sup eq1}
\tilde{u}_t =  F(D\tilde{u},D^2\tilde{u},t) + l(t) |D\tilde{u}| \ \ \hbox{ in } Q_T.
\end{align}
Similarly, let u be \textit{a viscosity subsolution} of \eqref{main eq2}. Then, the function
\begin{align*}
\hat{u}(x,t) := \sup_{y \in \overline{B}_{A-L(t)}(x)}u(y,t)
\end{align*}
is \textit{a viscosity subsolution} of 
\begin{align}
\label{inf eq1}
\hat{u}_t =  F(D\hat{u},D^2\hat{u},t)  - l(t) |D\hat{u}| \ \ \hbox{ in } Q_T.
\end{align}
\end{lem}

\begin{proof}
Let us show that the function $\tilde{u}$ is a viscosity supersolution of \eqref{sup eq1}, the subsolution part can be proved with parallel arguments.
For simplicity we will only present the proof for the case $l(t)=c > 0$, in which case $T=A/c$. 

\medskip

Suppose a test function $\phi$ touches $\tilde{u}_*  $ from below at $(x_0,t_0) \in Q_T$. It holds that 
\begin{align}
\label{eqn:supg11}
\tilde{u}_*(x_0,t_0)-\phi(x_0,t_0)= 0 \hbox{ and } \tilde{u}_*(x,t)-\phi(x,t) \geq 0 \hbox{ in } \N_{\delta}(x_0,t_0):=B_\delta(x_0) \times (t_0 - \delta, t_0]
\end{align}
for some $\delta > 0$. From the construction of $\tilde{u}_*$, there exists $x_1 \in \R^n$ such that
\begin{align}
\label{eqn:supg12}
|x_1 - x_0| \leq A-c t_0 \hbox{ and } \tilde{u}_*(x_0,t_0) = u_*(x_1,t_0).
\end{align}

\medskip

If $D\phi(x_0, t_0) = 0$, then it suffices to show that
\begin{align}
\label{eqn:supg211}
\phi_t(x_0,t_0) \geq  F_*(D\phi(x_0,t_0), D^2\phi(x_0,t_0),t_0).
\end{align}
We choose the shifted test function $\psi(x,t):=\phi(x - x_1 + x_0,t)$ and claim that $\psi$ touches $u_*$ from below at $(x_1,t_0)$. As $c > 0$, \eqref{eqn:supg12} yields that
\begin{align}
\label{eqn:supg22}
|x_1 - x_0| \leq A - ct \hbox{ for all } t \in (t_0 - \delta, t_0].
\end{align}
From \eqref{eqn:supg11}, we have $u_*(x_1,t_0)-\psi(x_1,t_0) = 0$. From we have $u_*(x_1,t_0)-\psi(x_1,t_0) = 0$ again and \eqref{eqn:supg22}, it holds that
\begin{align}
u_*(x,t)-\psi(x,t)  \geq \tilde{u}_*(x - x_1 + x_0,t)-\phi(x - x_1 + x_0,t) \hbox{ for any } (x,t) \in \N_{\delta}(x_1,t_0)
\end{align}
which yields the claim. Since $u_*$ is a viscosity supersolution of \eqref{main eq2} we have the corresponding PDE inequality for $\psi$ at $(x_1,t_0)$, which translates to \eqref{eqn:supg211}.

\medskip

Next, we suppose that $D\phi(x_0, t_0) \neq 0$. If $|x_1 - x_0| < A-c t_0$, then $u(\cdot, t_0)$ is constant in a small neighborhood of $x_0$ in $\R^n$ and it holds that $D\phi(x_0, t_0) = 0$. Thus, we have $|x_1 - x_0| = A-c t_0$. We claim that the shifted test function $\psi(x,t):=\phi(x-(A-ct)\vec{n},t)$ touches $u_*$ from below at $(x_1,t_0)$ where $$\vec{n}:=\frac{x_1-x_0}{|x_1-x_0|}.$$
First, note that $x_1 - (A - c t_0) \vec{n} = x_0$ and thus $u_*(x_1,t_0)-\psi(x_1,t_0)= 0$. Furthermore, if we choose $\e =\frac{1}{2} \min\left\{ \delta, t_0 \right\}$, then
\begin{align}
\label{eqn:supg23}
(x-(A-ct)\vec{n},t) = (x-x_1+x_0 + c(t-t_0) \vec{n},t) \in \N_{\delta}(x_0, t_0) \ \ \hbox{ for all } (x,t) \in \N_{\e}(x_1, t_0).
\end{align}
\eqref{eqn:supg11} and \eqref{eqn:supg23} imply that 
\begin{align*}
u_*(x,t)-\psi(x,t)  \geq \tilde{u}_*(x-(A-ct)\vec{n},t)-\phi(x-(A-ct)\vec{n},t) \geq 0 \ \ \hbox{ for all } (x,t) \in \N_{\e}(x_1, t_0),
\end{align*}
which yields the claim.

\medskip

As described in the first case, since $u_*$ is a viscosity supersolution of \eqref{main eq2} we have the corresponding PDE inequality for $\psi$ at $(x_1,t_0)$, which translates to
\begin{align}
\label{sup11}
\phi_t(x_0,t_0) + c D\phi(x_0,t_0) \cdot \vec{n} \geq  F_*(D\phi(x_0,t_0), D^2\phi(x_0,t_0),t_0).
\end{align}
Since $D\phi(x_0,t_0) \neq 0$ and the level set $\{x \in \R^n : \phi(x,t_0) = \phi(x_0,t_0)\}$ touches $\partial \Omega_{t_0}(\tilde{u})$ from inside at $x_0$,  $-D\phi(x_0,t_0)$ is parallel to the outward normal $\vec{n}$ of $\partial \Omega_{t_0}(\tilde{u})$ at $x_0$. Therefore, \eqref{sup11} yields
\begin{align*}
\phi_t(x_0,t_0) \geq  F_*(D\phi(x_0,t_0), D^2\phi(x_0,t_0),t_0) +c |D\phi(x_0,t_0)|.
\end{align*}
Now we can conclude that the function $\tilde{u}$ is viscosity supersolution of \eqref{sup eq1}.

\medskip

In general, we choose the shifted test function $\psi(x,t):=\phi(x - x_1 + x_0,t)$ or $\phi(x-L(t) \vec{n},t)$ and apply the parallel arguments to conclude.
\end{proof}

The following lemma will be used in section 3 to ensure uniform continuity of $\Omega_t(u)$ over time in Hausdorff distance.

\begin{lem}\label{lem_holder}
Let $u$ be a bounded viscosity solution of $\eqref{main eq2}$ or $\eqref{main_eq22}$ given by the form \eqref{eq421}. Then the following holds for $0<\delta <\frac{1}{{\|\eta\|_\infty}}$:  If $B_{2\delta}(x_0) \subset (\Omega_{t_0}(u))^C$ (or $\Omega_{t_0}(u)$), then  $B_{\delta}(x_0)\subset (\Omega_t(u))^C$(or $\Omega_t(u)$) for $t_0\leq t\leq t_0+\frac{\delta^2}{n}$. 
	
\end{lem}

\begin{proof}
We will verify the case when $B_{2\delta}(x_0)$ lies outside of $\Omega_{t_0}(u)$, since the rest follows from a parallel barrier argument. Let us compare $u$ with a radial barrier $\phi$ defined by $$\phi := - \chi_{B_{r(t)} (x_0)} + \chi_{{B_{r(t)} (x_0)}^C},$$ where $r:\left[t_0,t_0+\frac{\delta^2}{n}\right) \rightarrow \R$ solves $r(t_0):=2\delta, r'(t) := -\frac{n-1}{\delta} - {\|\eta\|_\infty}$. By assumption $u^*(x,t_0) \leq \phi_*(x,t_0)$.
	
\medskip
	
Let us show that $\phi$ is a viscosity supersolution for $t_0\leq t\leq t_0+\frac{\delta^2}{n}$. Since $\phi$ is a radial function, the normal velocity on $\partial \Omega_t(\phi)$ is equal to $-r'(t)$, and the mean curvature on $\partial \Omega_t(\phi)$ is $- \frac{n-1}{r(t)}$. 
Moreover, we have
\begin{align}
r'(t) &= - \dfrac{n-1}{\delta} - {\|\eta\|_\infty} \geq - \dfrac{n}{\delta}.
\end{align}
Since $r(t_0)=2\delta$, it follows that $r(t) \geq \delta$ if $t_0\leq t\leq t_0+\frac{\delta^2}{n}$. Therefore, it holds that for $t_0\leq t\leq t_0+\frac{\delta^2}{n}$ 
\begin{align}
-r'(t) = \frac{n-1}{\delta} + {\|\eta\|_\infty} \geq \frac{n-1}{r(t)} + \eta(t)
\end{align}
and we conclude. Now by Theorem \ref{comparison}(1), $u^* \leq \phi_*$ for $t_0\leq t\leq t_0+\frac{\delta^2}{n}$ and thus $B_{\delta}(x_0+\delta\nu)$ lies outside of $\Omega_t(u)$ for $t_0\leq t\leq t_0+\frac{\delta^2}{n}$. 
\end{proof}

\section{Geometry of the Flow}
\label{sec_pre}

In this section we study geometric properties of evolution of \eqref{main eq2}, following a strong notion of star-shapedness, $\rr$. This property, introduced in \cite{Feldman:2014hb},  is useful for problems which satisfy the reflection comparison principle (See Theorem~\ref{reflection comparison} below).

\subsection{Geometric Properties}

\begin{DEF}\label{starshaped} 
A bounded set $\Omega$ in $\R^n$ is \textit{star-shaped with respect to a ball} $B_r(0)$ if for any point $y \in B_r(0)$, $\Omega$ is star-shaped with respect to $y$. Let 
$$
S_r:=\{\Omega:\textit{star-shaped with respect to } B_r(0)\} \hbox{ and } S_{r,R}:=S_r\cap\{\Omega:\Omega\subset B_R(0)\}.
$$
\end{DEF}

The following lemma is immediate from the interior and exterior cone properties of sets in $S_r$.

\begin{lem}
\label{star2}
For a continuously differentiable and bounded function $\phi : \R^n \rightarrow \R$, let us denote the positive set of $\phi$ by $\Omega(\phi)$. Let us assume that $\Omega(\phi)$ contains $B_r(0)$ and $D\phi \neq 0$ on $\partial \Omega(\phi)$. Then 
the set $\Omega(\phi)$ is in $S_r$ if and only if 
$$x \cdot \vec{n}_x = x \cdot \left(- \frac{D\phi}{|D\phi|}(x)\right) \geq r  \hbox{ for all } x\in \partial\Omega(\phi), $$
where $\vec{n}_x$ denotes the outward normal of $\partial \Omega(\phi)$ at $x$.

\end{lem}
\begin{figure}[h]
	\centering
	\begin{subfigure}[t]{0.3\textwidth}
		\includegraphics[width=\textwidth]{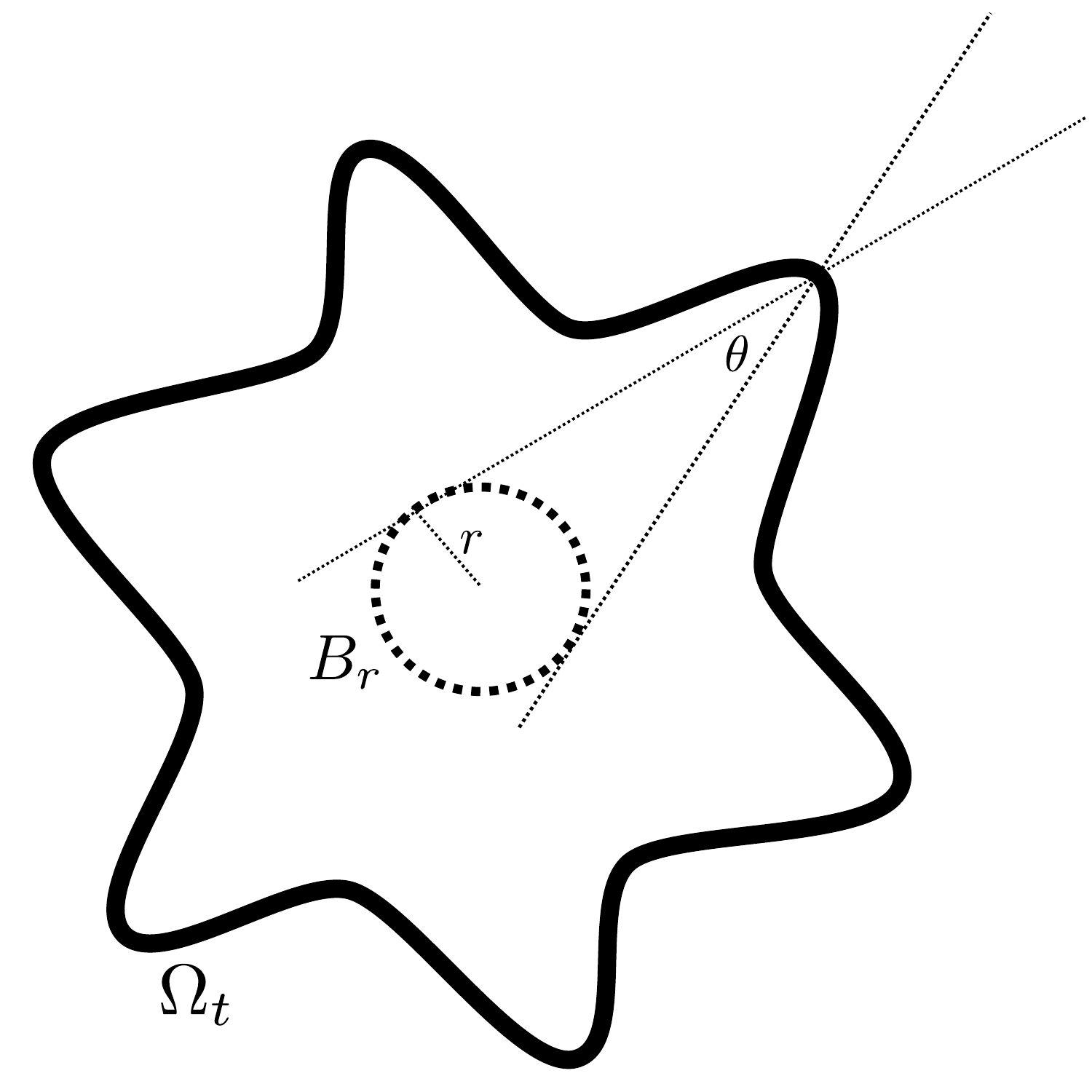}
		\caption{Star-shapedness}
	\end{subfigure}
	\hspace{2cm}
	\begin{subfigure}[t]{0.3\textwidth}
		\includegraphics[width=\textwidth]{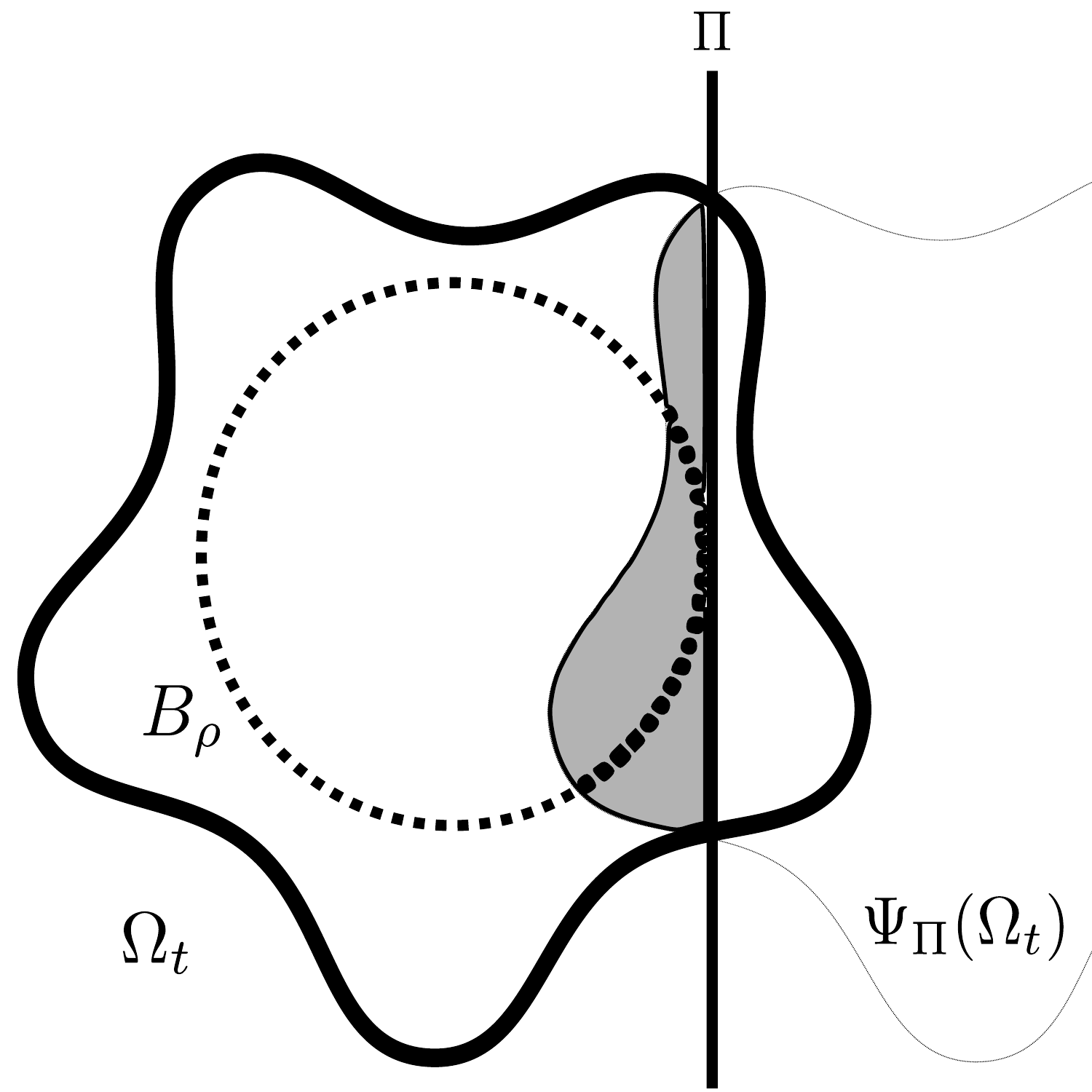}
		\caption{$\rho$-reflection}
	\end{subfigure}
	\caption{}
\end{figure}

Next we proceed to define the reflection property. For a hyperplane $\Pi = \Pi_{\nu}(s):=  \{x: x\cdot\nu = s\}$, let  $\Psi_{\Pi}$ denote the corresponding reflection, i.e.,
\begin{align}
\label{ref-fun}
\Psi_{\Pi(s)}(x) := x- 2 \langle x-s\nu, \nu \rangle \nu.
\end{align}
 
For $\Pi$ that doesn't contain the origin, we denote the half-spaces divided by $\Pi$ by $\Pi^+$ and $\Pi^-$, where $\Pi^-$ contains the origin.

\begin{DEF}
\label{rho-ref}\cite[Definition 10]{Feldman:2014hb}
 bounded, open set $\Omega$ has $\rr$ if 
\begin{itemize}
\item[(i)] $\Omega$ contains $\overline{B_\rho(0)}$ and
\item[(ii)] $\Omega$ satisfies, for all  $\nu \in S^{n-1}$ and all $s>\rho$,
$$\Psi_{\Pi_\nu(s)} (\Omega \cap \Pi^+_{\nu}(s)) \subset \Omega \cap \Pi^-_{\nu}(s).$$ 
\end{itemize}
\end{DEF}

The $\rr$ property can be viewed as a smallness condition on the Lipschitz norm distance between $\partial\Omega$ and the nearest ball (see the Appendix in \cite{Feldman:2014hb}). The following lemma states several properties and the relationship between the two concepts introduced above, $\rr$ and $S_r$ (See Figure 2, 3 and \cite[Figure 2]{Feldman:2014hb}). 

\begin{lem}\cite[Lemma 3, 9, 10, 24]{Feldman:2014hb} 
\label{star}
\begin{enumerate}
\item
For a bounded domain $\Omega$ containing $B_r(0)$, the followings are equivalent:

\noindent(i) $\Omega \in S_r$.

\noindent(ii) There exists $\e_0 = \e_0(r) >0$ such that
\begin{align}
\label{eqn_star}
\Omega \subset\subset \bigcap_{|z| \leq \delta \e} [(1+\e)\Omega +z] \hbox{ for all } 0<\e<\e_0 \hbox{ and } 0< \delta <r,
\end{align}
\noindent(iii) For all $x\in \Omega$, there is an interior cone to $\Omega$:
\begin{align}
\label{eqn_ic}
IC(x,r):=\left((x+C(-x,\theta_{x,r})) \cap C(x,\frac{\pi}{2} - \theta_{x,r} ) \right) \cup B_r(0) \subset \Omega \hbox{ for } |x| \geq r
\end{align}
where $C(x,\theta)$ is a cone in the direction $x$ with opening angle $\theta$ for $x \in \R^n$ and $\theta \in [0, \pi]$,
\begin{align}
\label{eqn:c}
C(x,\theta):=\{ y \mid \langle x,y \rangle > \cos \theta |x| |y| \} \hbox{ and } \theta_{x,r} := \arcsin \frac{r}{|x|} \in \left[0,\frac{\pi}{2}\right]. 
\end{align}

\begin{figure}[h]
	\centering
	\begin{subfigure}[t]{0.25\textwidth}
		\includegraphics[width=\textwidth]{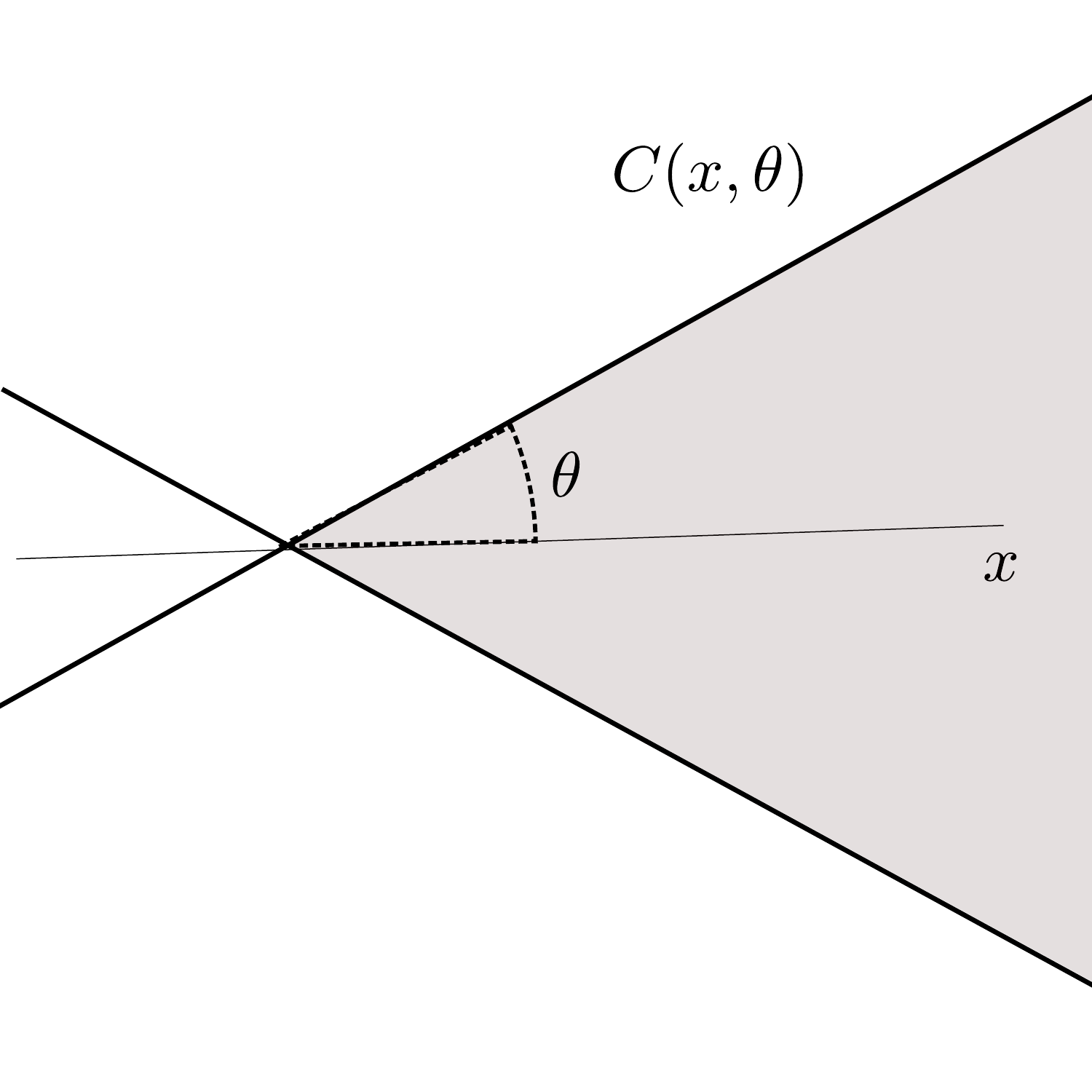}
		\caption{A cone in the direction $x$ with opening angle $\theta$}
	\end{subfigure}
	\hspace{2cm}
	\begin{subfigure}[t]{0.3\textwidth}
		\includegraphics[width=\textwidth]{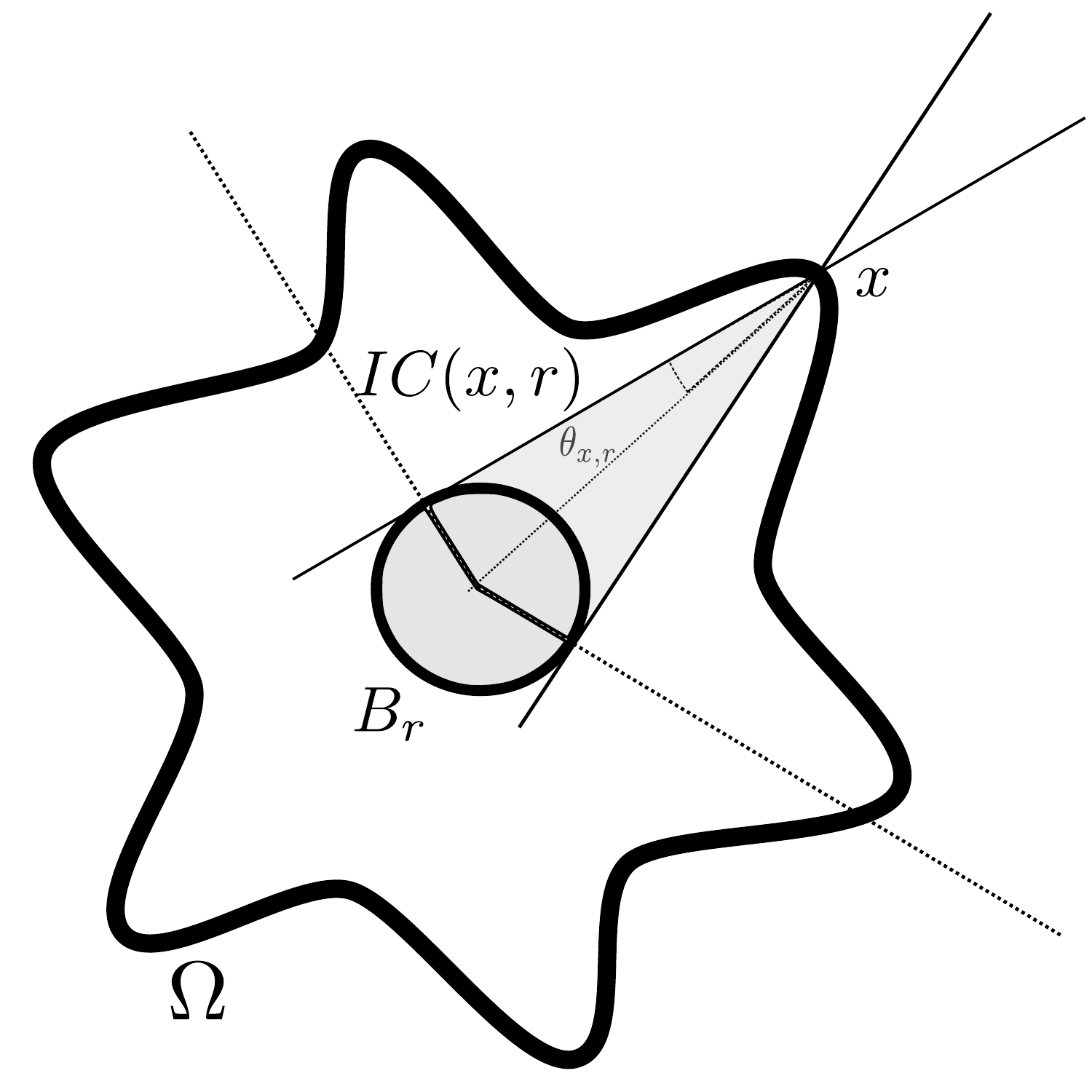}
		\caption{An interior cone to $\Omega$}
	\end{subfigure}
	\caption{}
\label{fig:ic}	
\end{figure}

\noindent (iv) For all $x\in  \Omega^C$, there is an exterior cone to $\Omega$:
\begin{align}
\label{eqn_ec}
EC(x,r):=x+C(x,\theta_{x,r}) \subset \Omega^C \hbox{ where } \theta_{x,r} = \arcsin \frac{r}{|x|} \in \left[0,\frac{\pi}{2}\right]. 
\end{align}
\item
Suppose that $\Omega$ has $\rr$.  Then $\Omega \in S_r$  with  
\begin{align}
\label{eqn_ra}
r=(\inf_{x\in \partial \Omega} |x|^2 - \rho^2 )^{1/2}. 
\end{align}
Moreover
\begin{align}
\label{eqn_4rho}
\sup_{x \in \partial \Omega} |x| - \inf_{x \in \partial \Omega} |x| \leq 4\rho.
\end{align}

\item
Suppose that $\Omega$ is in $S_{r,R}$. If  there exists $\rho>0$ such that $\overline{B_\rho(0)} \subset \Omega$ and $\rho^2 \geq 5(R^2 - r^2)$, then $\Omega$ has $\rr$.

\end{enumerate}
\end{lem}

\begin{thm}
\label{reflection comparison}
(Reflection Comparison)
Suppose that $\Omega_0$ has $\rr$. Let $u$ be a bounded viscosity solution of \eqref{main eq2} given by the form \eqref{eq421}. Let $\Pi$ be a hyperplane in $\R^n$ such that $\Pi \cap B_\rho(0)=\emptyset$. Then the reflected function $u(\Psi_{\Pi}(x),t)$ is also a bounded viscosity solution in $ \Pi^{-} \times (0,\infty)$. Moreover
\begin{align}\label{reflection}
\Psi_{\Pi}(\Omega_t \cap \Pi^{+}) \subset \Omega_t \cap \Pi^{-} \hbox{ for all } t>0 \hbox{ if true at } t=0.
\end{align}
\end{thm}
\begin{proof}

It is easy to see that $u(\Psi_{\Pi}(x),t)$ is also a viscosity solution of \eqref{main eq2} since $F$ is independent of $x$.

To show \eqref{reflection}, we will use the comparison principle in $\Pi^- \times [0,\infty)$.  To do so it is easier for us to consider a continuous version of $u$, i.e. let $\tilde{u}$ be the unique viscosity solution of \eqref{main eq2} with uniformly continuous initial data $\tilde{u}(x,0)$ defined by $\tilde{u}(x,0):= -\min\{  \sd(x, \Omega_0), 2R \},$
where $R$ is chosen large enough that $\Omega_0 \subset \subset B_R$. As $u$ is given by the form \eqref{eq421}, Theorem \ref{thm-invar} combined with the uniqueness implies that $\Omega_t(\tilde{u})$ is equal to $\Omega_t(u)$ for all $t>0$. 

\medskip

Note that Theorem~\ref{comparison}(2) implies that $\tilde{u}$ is uniformly continuous. As $\tilde{u}(\cdot, 0)$ is bounded in $\R^n$, we apply Theorem~\ref{comparison}(3) to conclude that $\tilde{u}$ is bounded in $\oQ$. Since $\tilde{u}(\Psi_{\Pi}(x),0) \leq \tilde{u}(x,0)$ in $\Pi^{-}$ and $\tilde{u}(\Psi_{\Pi}(x),0) = \tilde{u}(x,0)$ on $\Pi$, Theorem \ref{comparison}(1) applies to $\tilde{u}(x,t)$ and  $\tilde{u}(\Psi_{\Pi}(x),t)$ to yield $$\tilde{u}(\Psi_{\Pi}(x),t) \leq \tilde{u}(x,t)$$ for all $x \in \Pi^{-}$ and $t>0$. Therefore \eqref{reflection} follows.
\end{proof}

\begin{thm}
\label{rho zero bd}
Suppose that $\Omega_0$ has $\rr$. Let $u$ be a bounded viscosity solution of \eqref{main eq2} given by the form \eqref{eq421}. Let $I = [0,T)$ be the maximal interval satisfying  $\overline{B_\rho} \subset \Omega_t(u)$. Then, $\Omega_t(u)$ has $\rr$ for $t\in I$.
\end{thm}

\begin{proof}
From the definition of $\rr$, it is enough show that, for any unit vector $\nu$ in $\R^n$,
\begin{align}
\label{rho zero bd1}
\Psi_{\Pi_\nu(\rho)} (\Omega_t(u) \cap \Pi^+_{\nu}(\rho)) \subset \Omega_t(u) \cap \Pi^-_{\nu}(\rho) \hbox{ for } t\in I.
\end{align} 
Since $\Omega_0(u)$ has $\rr$, \eqref{rho zero bd1} holds at $t=0$, and we can conclude by Theorem~\ref{reflection comparison}.
\end{proof}

In the next section, we will show that  $\Omega_t(u)\in S_{r,R}$ in $[0,T]$ if it starts with some geometric restriction for the initial data. This leads to the following regularity of $\Omega_t(u)$ over time.

\begin{cor}\label{holder0}
Let $u$ be a bounded viscosity solution of $\eqref{main eq2}$ given by the form \eqref{eq421}. If $\Omega_t(u) \in S_{r, R}$ and $|\eta(t)| \leq K$ in $[0,T]$, then there exists $C=C(r,R,K,T)$ such that we have
\begin{align}
\label{eqn:1holder}
d_H(\partial \Omega_t(u), \partial \Omega_s(u)) \leq C(s-t)^{\frac{1}{2}} \hbox{ for all } 0 \leq t \leq s \leq T. 
\end{align}
\end{cor}

\begin{proof}
Choose $\delta \in \left( 0, \min \left\{ \frac{1}{K}, \frac{r}{2} \right\} \right)$ and $t \in [0,T]$. We claim that
\begin{align}
\label{eqn:holder11}
\sup_{x \in \partial \Omega_t(u)} d(x, \partial \Omega_s(u)) \leq \frac{2R \delta}{r} \hbox{ for all } s \in I:= \left[ t,  \min\left\{t+ \frac{\delta^2}{n}, T\right\} \right].
\end{align}
As $\Omega_t(u) \in S_{r, R}$, there exists $x_1=x_1(s) \in \partial \Omega_t(u)$ such that
\begin{align}
\label{eqn:holder12}
\sup_{x \in \partial \Omega_t(u)} d(x, \partial \Omega_s(u)) = d(x_1, \partial \Omega_s(u)) \hbox{ for } s \in I.
\end{align}
Let $y = \left(1 - \frac{2\delta}{r}\right)x$ and $z = \left(1 + \frac{2\delta}{r}\right)x$. From the interior and exterior cone properties in  Lemma~\ref{star}, it holds that
\begin{align*}
B_{2\delta}(y) \subset \Omega_t(u) \hbox{ and } B_{2\delta}(z) \subset \Omega_t(u)^C.
\end{align*}
As the assumption in Lemma~\ref{lem_holder} is satisfied, we conclude that $y \in \Omega_s(u)$ and $z \in \Omega_s(u)^C$ for all $s \in I$. As $\Omega_s(u) \in S_{r,R}$ in $I$, there exists $x_2 \in \partial \Omega_s(u)$ such that
\begin{align}
\label{eqn:holder14}
|x_1 - x_2| \leq \max \{ |x_1 - y|, |x_1 - z| \} \leq \frac{2R\delta}{r}.
\end{align}
\eqref{eqn:holder12} and \eqref{eqn:holder14} imply \eqref{eqn:holder11}. As $\Omega_s(u), \Omega_t(u) \in S_{r,R}$, we apply Lemma~\ref{lem:haus} to conclude \eqref{eqn:1holder}.
\end{proof}

\subsection{Preservation of the $\rr$ property} 
In this subsection, we suppose that there exists a viscosity solution $u$ of our original equation \eqref{main eq1} in the sense of Definition~\ref{def-sol}, and show the preservation of the $\rr$ property. As a consequence, star-shapedness of $\Omega_t(u)$ is preserved for all time. Existence of this solution will be given later in section~\ref{sec_ene} and \ref{sec_exi}. 

\begin{thm}
\label{pre-star}
Suppose that $\Omega_0$ has $\rr$. Assume that there exists a bounded viscosity solution $u$ given by the form \eqref{eq421} of \eqref{main eq1} and \eqref{initial}. Then $\Omega_t(u)$ has $\rr$ for all $t>0$. In particular there exists $r_1>0$ such that $\Omega_t$ is \textit{star-shaped with respect to a ball} $B_{r_1}(0)$ for all $t>0$.
\end{thm}

The proof of above theorem consists of Theorem~\ref{rho zero bd} and Lemma~\ref{Brho}. In Lemma~\ref{Brho}, we show that the maximal interval $I$ in Theorem~\ref{rho zero bd} is $[0,\infty)$.

\begin{lem}
\label{Brho}
Let $u$ and $\Omega_0$ be as given in above theorem. Then, there exists $a>0$ depending on $\Omega_0$ such that $B_{(1+a)\rho} \subset \Omega_t(u)$ for all $t>0$.
\end{lem}

\begin{proof}
Since $\Omega_0$ has $\rr$, $B_{(1+a)\rho} \subset \Omega_0$ for some $a>0$. Due to Assumption~\ref{hyp_a} and the continuity of $\lambda$, one can choose a small $a>0$ such that 
\begin{align}
\label{eqn_Brho0}
\lambda[|\Omega|] > \frac{n-1}{\rho} \hbox{ for sets contained in } B_{(5+a)\rho}. 
\end{align}

\medskip

Suppose that $B_{(1+a)\rho}$ is not contained in $\Omega_{t_*}(u)$ at some $t_* >0$. Then, there exists $t_0 \in (0,t_*)$ such that $\partial \Omega_t(u)$ touches from outside $\partial B_{(1+a)\rho}$ at $(x_0,t_0)$ for the first time. Then, by \eqref{eqn_4rho} in Lemma \ref{star}, we have $$\sup_{x \in \partial \Omega_{t_0}(u)}|x| \leq 4\rho + \inf_{x \in \partial \Omega_{t_0}(u)}|x| = (5+a)\rho,$$ and thus $\Omega_{t_0}(u)$ is contained in $B_{(5+a)\rho}$. Hence it follows from \eqref{eqn_Brho0} that 
\begin{align}
\label{Brho-3}
\lambda[|\Omega_{t_0}(u)|] >  \frac{n-1}{\rho} > H[B_{(1+a)\rho}].
\end{align}
where $H[B_{(1+a)\rho}]$ is the mean curvature of $\partial B_{(1+a)\rho}$.

\medskip

Consider $\phi(x):= - \left(\dfrac{|x|}{(1+a) \rho}\right)^2$. Note that \eqref{Brho-3} and $|x_0| = (1+a)\rho$ yield
\begin{align}
\label{Brho-4}
\nabla \cdot \left(\frac{D\phi}{|D\phi|}\right)(x_0) + \lambda[|\Omega_{t_0}(u)|] = - H[B_{(1+a)\rho}] + \lambda[|\Omega_{t_0}(u)|] > 0
\end{align} 
Hence $\psi(x,t):=\phi(x)$ is a strict subsolution of \eqref{main eq2}  with $\eta(t)=\lambda[|\Omega_t(u)|]$ in a small neighborhood of $(x_0,t_0)$. 

\medskip

On the other hand, we have $\psi \leq 0$ in $Q$ and $\psi \leq -1$ outside of $B_{(1+a)\rho}$. Recall that $u$ is given by the form \eqref{eq421}. As $\Omega_{t_0}(u)$ touches $B_{(1+a)\rho}$ at $(x_0,t_0)$ for the first time, $\psi$ touches $u_*$ from below at $(x_0, t_0)$ and we have
\begin{align}
\nabla \cdot \left(\frac{D\psi}{|D\psi|}\right)(x_0, t_0) + \lambda[|\Omega_{t_0}(u)|] \leq \psi_t(x_0, t_0) = 0 
\end{align}
and this contradicts to \eqref{Brho-4}.
\end{proof}

{\bf Proof of Theorem \ref{pre-star}:}
First note that $\Omega_t(u)$ has $\rr$ thanks to  Lemma~\ref{Brho} and Theorem~\ref{rho zero bd} applied to $u(x,t)$ and $\eta(t) = \lambda[|\Omega_t(u)|]$.  Moreover from \eqref{eqn_ra} in Lemma \ref{star}, $\Omega_t(u) \in S_r$ for 
\begin{align}
\label{eqn_pre_star}
r=\left(\inf_{x\in \partial \Omega} |x|^2 - \rho^2 \right)^{1/2} \geq r_1:=\rho(a^2 + 2a)^{1/2}.
\end{align}
Hence $\Omega_t(u)$ is \textit{star-shaped with respect to a ball} $B_{r_1}$ for all $t>0$.
\hfill$\Box$

\medskip

A particular consequence of Theorem~\ref{pre-star} is that $\partial\Omega_t(u)$ is a locally Lipschitz graph. This, in combination with Lemma~\ref{regularity}, yields that the evolution is indeed $C^{1,1}$:

\begin{cor}\label{regularity2}
Let $u$ and $\Omega_0$ be as in Theorem~\ref{pre-star}. Then $\Omega_t(u)$ has $C^{1,1}$-boundary for all $t>0$. In particular its principal curvatures are bounded by $O(1+1/\sqrt{t})$. 
\end{cor}

Next we note that, with the sublinear growth condition imposed on $\lambda$, $\Omega_t(u)$ is uniformly bounded in finite time.

\begin{lem}\label{bound00}
Let $u$ and $\Omega_0$ be as given in Theorem~\ref{pre-star}. Then, there exists $R_1=R_1(T)>0$ such that $\Omega_t(u) \subset B_{R_1}$ in $[0,T]$.
\end{lem}

\begin{proof}
By Assumption~\ref{hyp_a}, there exists a constant $C_1>0$ such that $\lambda[|B_R|] \leq C_1R$ for all $R\geq\rho$. Since $\Omega_0$ is bounded, there exists $\hat{R} > \rho$ such that $\Omega_0 \subset\subset B_{\hat{R}}$. Let us compare $u$ with a radial barrier $\phi : \oQ \to \R$ defined by $$\phi(x,t) :=  \chi_{B_{r(t)}}(x) - \chi_{{B^C_{r(t)}}}(x) \hbox{ for } (x,t) \in \oQ,$$ where $r:[0,T] \rightarrow \R$ is defined by $r(t) := \hat{R} e^{(C_1+1)t}$. Note that $\Omega_0(u) \subset \subset \Omega_0(\phi)$, and $r'(t) = (C_1 + 1) r(t)$.

\medskip

Choose $\e \in (0, \hat{R} C_1^{-1})$ and let us show that $\Omega_t(u) \subset  B_{r(t) + \e}$ for all time. Suppose it is false, then we have
\begin{align}
\label{bound00_11}
t_0:= \sup\{ t: \Omega_s(u) \subset B_{r(s) + \e} \hbox{ for } 0\leq s\leq t\} < +\infty.
\end{align}
By Corollary~\ref{holder0}, $\partial\Omega_t(u)$ evolves continuously in time and thus
\begin{align}
\label{bound00_111}
\partial \Omega_s(u) \cap \partial B_{r(s) + \e} \neq \emptyset.
\end{align}
Combining \eqref{bound00_11} with Lemma~\ref{Brho}, we have $|B_\rho| \leq |\Omega_t(u)| \leq |B_{r(t) + \e}|$ in $[0,t_0]$. Furthermore, as $ r(t) \geq \hat{R} > \rho$, it holds that
\begin{align}
\lambda[|\Omega_t(u)|] \leq C_1(r(t) + \e) \leq r'(t) + \frac{n-1}{r(t)}.
\end{align}
Therefore, $\phi$ is a viscosity supersolution of \eqref{main eq2} with $\eta(t)=\lambda[|\Omega_t(u)|]$ in $[0,t_0]$. Note that $u^* \leq \phi_*$ at $t=0$. From Theorem \ref{comparison}(1) we have $u^* \leq \phi_*$ in $[0, t_0)$ and thus
\begin{align}
\overline{\Omega}_t(u) \subset B_{r(t)} \hbox{ in } [0,t_0).
\end{align}
By Corollary~\ref{holder0} again, $\partial\Omega_t(u)$ evolves continuously in time and thus we have $\Omega_{t}(u) \subset B_{r(t_0)}$ in $[0,t_0]$, which contradicts \eqref{bound00_111}.

As a consequence, we conclude that 
\begin{align}
\label{eqn_bound00}
\Omega_t(u) \subset B_{R_1} \hbox{ where } R_1(T):= \hat{R}e^{(C_1+1) T} + \e 
\end{align}
in $[0,T]$.
\end{proof}

\medskip

We finish this section with some properties of our solutions that will be used later. The following corollary holds due to the fact that $\Omega_0$ has $\rr$ and therefore for small $\e>0$ the sets $\Omega_0^{\e,+}:=(1+\e)\Omega_0$ and $\Omega_0^{\e,-}:=(1+\e)^{-1}\Omega_0$ satisfy $\rho (1+O(\e))$-reflection.

\begin{cor}\label{pre-star2}
Let $u$, $\Omega_0$ and $r_1$ be as given in Theorem~\ref{pre-star} and $R_1$ as given in Lemma~\ref{bound00}. Then for sufficiently small $\e>0$ viscosity solutions $u^\pm$ of \eqref{main eq1} starting from $\Omega_0^{\e,\pm}$ have their positive sets $\Omega_t(u^\pm)$ in $S_{r_1 -O(\e),R_1 + O(\e)}$ in $[0,T]$.
\end{cor}

\begin{lem}
\label{cor_holder}
Let $u$, $\Omega_0$ and $r_1$ be as given in Theorem~\ref{pre-star} and $R_1$ as given in Lemma~\ref{bound00}. Then, there exists positive constants $\tilde{K}_{\infty} = \tilde{K}_{\infty}(r_1,R_1,T)$ and ${\tilde{K}_{1/2}}~=~{\tilde{K}_{1/2}}(r_1,R_1,T)$ such that the following holds for all $t,s$ in $[0,T]:$
\begin{align}
\label{eqn:1hor}
\Big| \lambda[|\Omega_{t}(u)|] - \lambda[|\Omega_{s}(u)|] \Big| \leq {\tilde{K}_{1/2}} |t-s|^{\frac{1}{2}}
\end{align}
and
\begin{align*}
\Big| \lambda[|\Omega_{t}(u)|] \Big| \leq \tilde{K}_{\infty}.
\end{align*}
\end{lem}

\begin{proof}	
From Lemma~\ref{Brho}~and~\ref{bound00}, $|\Omega_t|$  is bounded away from zero and infinity, and thus $\lambda$ is bounded. Next, by the Lipschitz continuity of $\lambda$ and the last inequality of \eqref{eqn:1cpt} in Lemma~\ref{lem-cpt}, there exists $C_1(r_1,R_1,T)$ such that
\begin{align*}
\Big| \lambda[|\Omega_t(u)|] - \lambda[|\Omega_s(u)|] \Big| \leq  C_1 d_H(\Omega_t(u), \Omega_s(u)) \hbox{ for } t,s \in [0,T].
\end{align*}
From the above inequality and H{\"o}lder continuity in Corollary~\ref{holder0}, we conclude \eqref{eqn:1hor}.
\end{proof}

Finally, let us show Lipschitz continuity of $|\Omega_t|$ in time for the later purpose in Lemma~\ref{lem_uni_1}.
\begin{lem}\label{lem_lip}
Let $u$, $\Omega_0$ and $r_1$ be as given in Theorem~\ref{pre-star}, $R_1$ as given in Lemma~\ref{bound00}, and $\tilde{K}_\infty$ as given in Lemma~\ref{cor_holder}. Then there exists $C=C(r_1,R_1,\tilde{K}_\infty)$ such that we have
\begin{align}
\label{eqn_lem_lip}
 \left| |\Omega_t(u)| - |\Omega_s(u)| \right| \leq C\left(1+\frac{1}{\sqrt{t}}\right)|s-t|\hbox{ for }0 \leq t \leq s \leq T.
\end{align}
\end{lem}

\begin{proof}
First, by Corollary~\ref{regularity2}, all principal curvatures are bounded by $M(t) := C_1(1+1/\sqrt{t})$ for some constant $C_1=C_1(r_1,R_1,\tilde{K}_\infty)$. 
Thus, there exist interior and exterior balls of radius $M(t)^{-1}$ on each point of $\partial \Omega_t(u)$ for all $t>0$. As described in Corollary~\ref{holder0}, we apply Lemma~\ref{lem_holder} in these balls to conclude that
\begin{align*}
d_H(\partial \Omega_t(u), \partial \Omega_s(u)) \leq C_2\left(1+\frac{1}{\sqrt{t}}\right)|s-t|\hbox{ for }0 \leq t \leq s \leq T.
\end{align*}
for some $C_2 = C_2(r_1,R_1, \tilde{K}_\infty)$.
Recall from the first and last inequalities of \eqref{eqn:1cpt} in Lemma~\ref{lem-cpt} that the volume difference is bounded by the Hausdorff distance. Thus, we conclude that there exists $C=C(r_1,R_1,\tilde{K}_\infty)$ satisfying \eqref{eqn_lem_lip}.
\end{proof}

\section{Uniqueness of the Flow} 
\label{sec_uni}

In this section, we show the uniqueness for solutions of \eqref{main eq1} and \eqref{main eq2} with given initial data \eqref{initial}. As pointed out in Remark~\ref{rem_uni}, the comparison principle (Theorem 2.2) does not deliver the uniqueness for a discontinuous viscosity solution, due to the possible fattening phenomena of level sets. We show that our flow \eqref{main} can be uniquely determined when the initial data has $\rr$. We follow the argument of \cite{BCCN09}, where the uniqueness result is shown for convex evolution of volume-preserving flow.  

\medskip

In section~\ref{sec_uni_1}, we show the short-time uniqueness result for  \eqref{main eq2} in Theorem~\ref{thm-uni-2} for a star-shaped initial data $\Omega_0$. We define appropriate convolutions to perturb solutions (see Definition~\ref{def_cp}) and show that our perturbation preserves sub- and supersolution properties for \eqref{main eq2}. These perturbations are more delicate than those used in \cite{Gig06} due to the presence of the time-dependent forcing $\eta$. We use these perturbations to obtain the uniqueness results. At this point, it is crucial to find a uniform interval $[0,t_1]$ where these convolutions are well defined in this interval (see Lemma~\ref{lem_cp_1}). It remains open whether the flow \eqref{main eq2} stays unique beyond the interval.

\medskip

In section~\ref{sec_uni_2}, we show the global-time uniqueness for \eqref{main eq1}  when its initial data has $\rr$ (see Theorem~\ref{cor-uni-2}). Here we know that any evolution, if exists, preserves the $\rr$ property, which we use to iterate the short-time uniqueness result from the previous subsection. The key step is to estimate the difference between $\lambda[|\Omega_t(u)|]$ and $\lambda[|\Omega_t(v)|]$ for two possible solutions (see Lemma~\ref{lem_uni_1}).

\subsection{Short-time uniqueness of (2.1)}
\label{sec_uni_1}

\begin{DEF}
\label{def_reg}
\cite[Definition 2.1]{Barles:1993gaba} For a function $u : \oQ \to \R$ and $t \geq 0$, we say that $\Omega_t(u)= \{u(\cdot,t)>0\}$ is {\it  regular} if the closure of $\Omega_t(u)$ is $\{  x \in \R^n : u(x,t) \geq 0\}$, and the interior of $\{  x \in \R^n : u(x,t) \geq 0\}$ is $\Omega_t(u)$.
\end{DEF}

Note that for $t \geq 0$, if $\Omega_t(u)$ is regular, then the interface $\{x \in \R^n : u(x,t) = 0\}$ has an empty interior. 

\begin{lem}
\label{lem_cp_3}
\cite[Theorem 2.1]{Barles:1993gaba}
Let $u : \oQ \to \R$ be a viscosity solution of \eqref{main eq2} and \eqref{initial}. Then, $ \Omega_t(u)$ is regular for all $t \geq 0$ if and only if there exists a unique solution in $\oQ$ of \eqref{main eq2} with initial data $u(x,0) = u_0(x) := \chi_{\Omega_0} - \chi_{\Omega_0^C}$.
\end{lem}

Recall from section~\ref{sec_vis} that 
\begin{align}\label{upper_bd}
K_\infty :=\| \eta \|_{L^\infty([0,\infty))}.
\end{align}  We define $t_1=t_1(r,K_\infty)$ by
\begin{align}
\label{time_1}
t_1 := \frac{r}{10K_\infty} 
\end{align}
and we will show the following theorem in this section.

\begin{thm}\label{thm-uni-2}
Suppose that the initial set $\Omega_0$ is in $S_{r}$. Then, there is exactly one bounded viscosity solution $u$ of \eqref{main eq2} and \eqref{initial} in $[0,t_1]$ where $t_1$ is given in \eqref{time_1}. Moreover, $\Omega_t(u)$ is regular in $[0,t_1]$.
\end{thm}

We begin the proof with some definitions.

\begin{DEF}\label{def_cp}
For $\e, r>0$ and $L: [0, +\infty)\to \R$, let us define a maximal time $T_1=T_1(\e,r,L)$ by
\begin{align}\label{time}
T_1 := \sup \{ s >0 : L(t) < r\e/2 \hbox{ for all } t \in [0,s] \};
\end{align}
\begin{align*}
{\underline{u}}(x,t;\e,r,L) := \left. \inf \left\{  u\left(\frac{y}{1+\e},\frac{t}{(1+\e)^2}\right) \,\right| \, \hbox{ } y \in \overline{B_{r\e/2 -  L(t)}(x)} \right\};
\end{align*}
and
\begin{align*}
{\overline{u}}(x,t;\e,r,L) := \left. \sup \left\{  u\left(\frac{y}{1-\e},\frac{t}{(1-\e)^2}\right) \,\right| \, \hbox{ } y \in \overline{B_{r\e/2 - L(t)}(x)} \right\}
\end{align*}
\end{DEF}

\begin{lem} 

\label{lem_cp_0}
Let $u$ be a bounded viscosity solution of \eqref{main eq2} and \eqref{initial} with forcing $\eta$ and $\Omega_0 \in S_{r}$, and let $\eta_\e(t):= (1+\e)^{-1} \eta(t/(1+\e)^2)$. Let $\underline{u}$ and $\overline{u}$ be as given above with $L\in C^1([0,\infty))$.  Then the following holds in $(0,T_1)$ in the sense of viscosity solutions:
\begin{align}
\label{eq:cp_0}
\dfrac{{\underline{u}}_t}{|D{\underline{u}}|}(x,t) \geq  \nabla \cdot \left(\dfrac{D{\underline{u}}}{|D{\underline{u}}|}\right)(x,t) + \eta_{\e}(t) + L'(t)
\end{align}
and
\begin{align}
\label{eq:cp_01}
\dfrac{{\overline{u}}_t}{|D{\overline{u}}|}(x,t) \leq  \nabla \cdot \left(\dfrac{D{\overline{u}}}{|D{\overline{u}}|}\right)(x,t) + \eta_{-\e}(t) - L'(t).
\end{align}
Moreover, if $\e \leq \e_0(r)$ for $\e_0(r)$ given in \eqref{eqn_star}, we have 
\begin{align}
\label{eq:cp_1}
\Omega_0(\overline{u}) \subset\subset \Omega_0(u) \subset\subset \Omega_0(\underline{u}). 
\end{align}
\end{lem}

\begin{proof}
First, let us denote $v(x,t):=u\left(\frac{x}{1+\e},\frac{t}{(1+\e)^2}\right)$. Then, $v$ is a viscosity solution of 
\begin{align*}
\dfrac{v_t}{|Dv|}(x,t) =  \nabla \cdot \left(\dfrac{Dv}{|Dv|}\right)(x,t) +  \eta_\e(t).
\end{align*}
and thus Lemma~\ref{sup convolution_gen} implies \eqref{eq:cp_0}. Parallel arguments holds for $\overline{u}$.
	
\medskip

On the other hand, if $\Omega_0(u)$ is in $S_{r}$ then  Lemma~\ref{star} yields, for all $\e \leq \e_0(r)$,
\begin{align}
\label{eqn_set_inf}
\Omega_0(u) \subset\subset \bigcap_{|z| \leq r \e/2} [(1+\e)\Omega_0(u) +z] = \Omega_0({\underline{u}}),
\end{align}
and
\begin{align}
\label{eqn_set_sup}
\Omega_0({\overline{u}}) =  \bigcup_{|z| \leq r \e/2} [(1-\e)\Omega_0(u) +z]  \subset \subset \Omega_0(u).
\end{align}
\end{proof}

\begin{lem}
\label{lem_cp_1}
Let $\eta$ and $\eta_\e$ be as given in Lemma~\ref{lem_cp_0}, and let $t_1=r/10K_{\infty}$ be as given in \eqref{time_1}.  Then for the choice of $L(t) =\int_0^t -\eta_\e(s) +\eta(s) ds$ or $L(t) = \int_0^t \eta_{-\e}(s) -\eta(s)ds$ and for $0<\e\leq 1/4$, we have 
$$T_1=T_1(\e,r,L) \geq t_1 \hbox{ for } 0<\e<1/4.$$
\end{lem}

\begin{proof}
1. First, let us choose $L(t) =\int_0^t -\eta_\e(s) +\eta(s) ds$ and estimate the function $L$ by the change of variables.
\begin{align*}
L(t) &= \int_{0}^{t} \eta(s) - \frac{1}{1+\e} \eta\left(\frac{s}{(1+\e)^2}\right) ds,\\
&= \int_{0}^{t} \eta(s) ds - (1+\e) \int_{0}^{ \frac{t}{(1+\e)^2}}  \eta(s) ds,\\
&= \int_{\frac{t}{(1+\e)^2}}^{t} \eta(s) ds - \e \int_{0}^{\frac{t}{(1+\e)^2} }  \eta(s) ds.
\end{align*}
	
Therefore, we conclude that for $\e\in (0,1/4)$
\begin{align}
\label{eqn_lem_cp}
|L(t)| 	&\leq K_\infty  t \left( \frac{\e^2 + 2\e}{(1+\e)^2} \right) + K_\infty \e t < 5K_\infty \e t.
\end{align}
	
\medskip
	
2. Similarly, let us choose 	$L(t) = \int_0^t \eta_{-\e}(s) -\eta(s)ds$, then  for $\e\in (0, 1/4)$
\begin{align*}
|L(t)| 	&= \Big| \int^{\frac{t}{(1-\e)^2}}_{t} \eta(s) ds - \e \int_{0}^{\frac{t}{(1-\e)^2} }  \eta(s) ds \Big|,\\
&\leq K_\infty  t \left( \frac{2\e - \e^2}{(1-\e)^2}\right) + K_\infty \e t \frac{1}{(1-\e)^2} < 5 K_\infty \e t .
\end{align*}

\medskip
	
3. By definition of $T_1$ we have $L(T_1) = r\e/2$. Thus  $5K_{\infty} \e t_1= r\e/2 = L(T_1) < 5 K_\infty \e  T_1 $.
\end{proof}

Lemma~\ref{lem_cp_0} and Lemma~\ref{lem_cp_1} imply the following.

\begin{lem}
\label{lem_cp_2} 
Let $u$ and $\Omega_0$ be as given in Lemma~\ref{lem_cp_0} and let $0<\e\leq \e_0(r)$. For $t_1$  given in \eqref{time_1}, ${\underline{u}}$ with the choice of  $L(t) = \int_0^t -\eta_{\e} + \eta$ is a viscosity supersolution of \eqref{main eq2} in $(0,t_1]$. Similarly, ${\overline{u}}$ with $L(t) = \int_0^t \eta_{-\e} - \eta$ is a  subsolution of \eqref{main eq2}  in $(0,t_1]$. Moreover it holds that $\overline{u} \leq u \leq \underline{u}$ in $[0,t_1]$.
\end{lem}

\begin{proof}
By Lemma~\ref{lem_cp_1}, ${\underline{u}}$ and $\overline{u}$ are well-defined in $[0,t_1]$. So, we could apply Lemma~\ref{lem_cp_0} and comparison principle in Theorem~\ref{comparison}(1) for \eqref{main eq2}  in $[0,t_1]$ to  conclude. 
\end{proof}

{\bf Proof of Theorem \ref{thm-uni-2}:} 
Suppose that $u$ and $v$ are two bounded solutions of \eqref{main eq2} and $u(\cdot,0) = v(\cdot,0)$ in $\R^n$. Let us construct ${\underline{u}}$ and $\overline{u}$ as in Lemma~\ref{lem_cp_2}. As $\Omega_0(\overline{u}) \subset\subset \Omega_0(v) = \Omega_0(u) \subset\subset \Omega_0(\underline{u})$ from \eqref{eq:cp_1}, we have $\overline{u}^*(\cdot, 0) \leq v_*(\cdot, 0)$ and $v^*(\cdot, 0) \leq \underline{u}_*(\cdot, 0)$ in $\R^n$. By Lemma~\ref{lem_cp_2} and the comparison principle in Theorem~\ref{comparison}(1), it holds that $\overline{u} \leq v \leq \underline{u}$ in $[0,t_1]$. Sending $\e$ to zero, we conclude that $u = v$ in $[0,t_1]$.
\hfill$\Box$

\medskip

Lastly, for the next subsection let us state the following lemma.

\begin{lem}
\label{lem_cp_3}
Let $u$ and $\Omega_0$ be as given in Lemma~\ref{lem_cp_0}. Then for $0<\e\leq \e_0(r)$ and $0\leq t\leq t_1$ we have  
$$(1-\e) \Omega_{t/(1-\e)^2}(u) \subset  \Omega_t(u) \subset (1+\e) \Omega_{t/(1+\e)^2}(u).$$
where $t_1$ is given \eqref{time_1}.
\end{lem}	
	
\begin{proof}
Lemma~\ref{lem_cp_2} implies that $\Omega_t(\overline{u}) \subset \Omega_t(u) \subset \Omega_t(\underline{u})$ in $[0,t_1]$. Moreover we have, by definition,
\begin{align*}
(1-\e) \Omega_{t/(1-\e)^2}(u) &\subset \Omega_t(\overline{u}), \hbox{ and } 
\Omega_t(\underline{u}) \subset (1+\e) \Omega_{t/(1+\e)^2}(u).  
\end{align*}
\end{proof}

\subsection{Uniqueness of mean curvature flows of forcing}
\label{sec_uni_2}

In this subsection, we show the uniqueness of our original equation $\eqref{main eq1}$. Here is the main theorem of this subsection.

\begin{thm}
\label{cor-uni-2}
Suppose that $\Omega_0$ has $\rr$. Then, there exists at most one bounded viscosity solution of \eqref{main eq1} and \eqref{initial}.
\end{thm}
	
Let $u$ and $v$ be two bounded viscosity solutions of \eqref{main eq1} and \eqref{initial}, and let $\eta(t;u):=\lambda[|\Omega_t(u)|]$ and $\eta(t;v):=\lambda[|\Omega_t(v)|]$. Fix $T>0$. Recall from Theorem~\ref{pre-star} and Lemma~\ref{bound00} that both $\Omega_t(u)$ and $\Omega_t(v)$ are in $S_{r_1,R_1}$ in $[0,T]$ where $r_1$ and $R_1$ are given in \eqref{eqn_pre_star} and \eqref{eqn_bound00}, respectively. From Lemma~\ref{cor_holder} that there exists a uniform bound of $\eta(t;u)$ and $\eta(t;v)$ in $[0,T]$,
\begin{align}\label{constant_K}
\tilde{K}_{\infty}:= \||\eta(t;u)|+|\eta(t;v)|\|_{L^\infty([0,T])}<\infty.
\end{align}

Recall $\eta_\e(t):= (1+\e)^{-1} \eta(t/(1+\e)^2)$ and define 
\begin{align}
\label{l1l2}
L_1(t) := \int_0^t -\eta_\e(s;u) +\eta(s;v)ds \hbox{ and } L_2(t) := \int_0^t \eta_{-\e}(s;u) -\eta(s;v) ds 
\end{align}

\begin{DEF}
For $\e \in  (0,\frac{1}{4})$, let us define 
\begin{align}
\label{time_t}
\tilde{T}_1 = \tilde{T}_1(\e, r_1, L_1, L_2) := \sup \left\{ s \in (0,T] : L_1(t) , L_2(t) < \frac{r_1\e}{2} \hbox{ for all } t \in [0,s] \right\}
\end{align}
where $r_1$ is given in \eqref{eqn_pre_star}. Remind that $r_1$ is chosen so that $\Omega_t(u)$ and $\Omega_t(v)$ are in $S_{r_1,R_1}$ for all $t \in [0,T]$.
\end{DEF}

Let $\underline{u}=\underline{u}(\cdot;\e,r_1,L_1)$ and $\overline{u}= \overline{u}(\cdot;\e,r_1,L_2)$ be as given  in Definition~\ref{def_cp}. 
The construction of $L_1$ and $L_2$ and Lemma~\ref{lem_cp_0} readily yields the following lemma.

\begin{lem}
\label{lem_uni_0}
${\underline{u}}$ and ${\overline{u}}$ are a viscosity supersolution, and subsolution, respectively, of \eqref{main eq2} with $\eta = \eta(;v)$  in $(0,\tilde{T}_1)$. Moreover, it holds that $\overline{u} \leq v \leq \underline{u}$ in $[0,\tilde{T}_1]$. Here, $\tilde{T}_1$ is given in \eqref{time_t}.
\end{lem}

\begin{lem}\label{lem_uni_1}
There exists $t_2>0$ such that for  any $\e \in (0,\frac{1}{4})$, 
\begin{align}
\label{time_t2}
\tilde{T}_1=\tilde{T}_1(\e, r_1, L_1, L_2) >t_2 
\end{align}
where $\tilde{T}_1$ is given in \eqref{time_t}.
\end{lem}
	
\begin{proof}
Let $t_1(r_1,\tilde{K}_\infty) =\frac{r_1}{5\tilde{K}_\infty}$ be as given in \eqref{time_1}. If $\tilde{T}_1\geq t_1$ for all $\e \in (0,\frac{1}{4})$, we take $t_2=t_1$. If $\tilde{T}_1< t_1$ for some $\e \in (0,\frac{1}{4})$, Lemma~\ref{lem_cp_3} implies that in $[0,\tilde{T}_1)$
\begin{align}
\label{eqn_lem_uni_1}
(1-\e) \Omega_{t/(1-\e)^2}(u) \subset  \Omega_t(u) \subset (1+\e) \Omega_{t/(1+\e)^2}(u).
\end{align}
	 
\medskip
	
Lemma~\ref{lem_uni_0} implies that $\Omega_t(\overline{u}) \subset \Omega_t(v) \subset \Omega_t(\underline{u})$ in $[0,\tilde{T}_1)$. Thus as shown in Lemma~\ref{lem_cp_3}, the following holds for  $0\leq t< \tilde{T}_1$:
\begin{align}
\label{eqn_lem_uni_2}
(1-\e) \Omega_{t/(1-\e)^2}(u) \subset \Omega_t(\overline{u}) \subset \Omega_t(v) \subset \Omega_t(\underline{u}) \subset (1+\e) \Omega_{t/(1+\e)^2}(u). 
\end{align}
		
By subtracting $\eta(s;u)$ and adding the same term, 
\begin{align}\label{eqn_lem_uni_4}
L_1(t) = \int_{0}^{t} \eta(s;v) - \eta_\e(s;u) ds= \int_{0}^{t} \eta(s;v) - \eta(s;u) ds + \int_{0}^{t} \eta(s;u) -  \eta_\e(s;u) ds. 
\end{align}
As Lemma~\ref{lem_cp_1}, the second term is bounded by $5\tilde{K}_{\infty} \e t$. As for the first term, from Lipschitz continuity of $\lambda$ for some $C_1>0$,
\begin{align*}
\mathcal{I}_1 := \left| \int_{0}^{t} \eta(s;v) - \eta(s;u)  ds \right| &\leq  \int_{0}^{t}  \Big|\lambda[|\Omega_s(v)|] - \lambda[|\Omega_s(u)|] \Big|ds  \leq  C_1\int_{0}^{t}  \Big| |\Omega_s(v)| - |\Omega_s(u)| \Big|ds 
\end{align*}
By \eqref{eqn_lem_uni_1}-\eqref{eqn_lem_uni_2} and Lemma~\ref{bound00}, 
\begin{align*}
\mathcal{I}_1 &\leq  C_1\int_{0}^{t} \Big|  \big| (1-\e) \Omega_{s/(1-\e)^2}(u) \big| -   \big| (1+\e) \Omega_{s/(1+\e)^2}(u) \big| \Big| ds  \\
&\leq  C_1 \int_{0}^{t} \Big| \big| \Omega_{s/(1-\e)^2}(u)\big| - \big| \Omega_{s/(1+\e)^2}(u)\big| \Big|ds  + C_2 \e t
\end{align*}
for some constant $C_2 = C_2(R_1, T)$. By Lemma~\ref{lem_lip}, we conclude that $\mathcal{I}_1$ is bounded by $C_3 \e t$ for some constant $C_3 = C_3(r_1,R_1,T, \tilde{K}_\infty)$. Therefore, we have $L_1(t) < (C_3 + 5\tilde{K}_{\infty}) \e t$ in $[0,\tilde{T}_1]$. By similar arguments, the bound holds for $L_2$ as well in $[0,\tilde{T}_1]$.
	
\medskip
	
Finally, by continuity of $L_1$ and $L_2$, we have $L_1(\tilde{T}_1) = r_1\e/2$ or $L_2(\tilde{T}_1) = r_1\e/2$. In both cases, it holds that
\begin{align*}
r_1\e/2 = L_1(\tilde{T}_1) (\hbox{ or } L_2(\tilde{T}_1)) < (C_3 + 5\tilde{K}_{\infty})    \tilde{T}_1 \e,
\end{align*}
so we conclude with 
\begin{align}
\label{eqn_lem_uni_5}
\tilde{T}_1 \geq t_2 = t_2(r_1,R_1,T,\tilde{K}_\infty) := \frac{r_1}{2 C_3 + 10 \tilde{K}_{\infty}}. 
\end{align}
\end{proof}
	
{\bf Proof of Theorem \ref{cor-uni-2}:} The first part is parallel to the proof of Theorem~\ref{thm-uni-2}. Let $u$ and $v$ be two viscosity solutions of \eqref{main eq1} and \eqref{initial}. By Lemma~\ref{lem_uni_0} and Lemma~\ref{lem_uni_1}, it holds that $\overline{u} \leq v \leq \underline{u}$ in $[0,t_2]$ where $t_2$ is given in \eqref{time_t2}. We can now send $\e$ to zero to conclude that $u = v$ in $[0,t_2]$.

\medskip

Next let us consider the corresponding convolutions of $\underline{u}$ and $\overline{u}$ in the time interval $t_0 + [0,t_2] \subset [0,T]$ for $t_0>0$ and $t_2$ given in \eqref{time_t2}. Note that $t_2$ given in \eqref{time_t2} does not depend on $t_0$ because both $\Omega_t(u)$ and $\Omega_t(v)$ are in $S_{r_1, R_1}$ for all $t \in [0,T]$. Thus, we can iterate the step 1 for $t_0 = kt_2$ on $k t_2 + [0,t_2]$, $k \in \mathbb{N}$ and, conclude that $u = v$ in $[0,T]$.
\hfill$\Box$

\section{Construction of flat flows}
\label{sec_ene}

In this section, we construct a flat flow for \eqref{main eq1}, which coincides our notion of viscosity solutions. Our approach is based on minimizing movements first introduced by Almgren-Taylor-Wang \cite{Almgren:1993gw} (see also  \cite{LS95}, \cite{Chambolle:2004ds}, \cite{BCCN09}).

\medskip

As in \cite{Feldman:2014hb}, we introduce a gradient flow with geometric constraint, corresponding to the preservation of star-shapedness obtained in Theorem~\ref{pre-star}. Our constraint is crucial to ensure the strong (in Hausdorff distance) convergence of the minimizing movements, which enables geometric analysis of the limiting flow. On the other hand the constraint also poses technical challenges when we show the coincidence of flat flows with viscosity solutions (See Proposition \ref{cor_bar} and Corollary \ref{cor_bar2}). This is why we first approximate our original problem with a ``restricted" version \eqref{main_eq22}.

\subsection{Constrained Minimizing Movements}
Recall the following energy functional associated with \eqref{main eq1},
\begin{align}
\label{energy2}
J(E) = \mathrm{Per}(E) - \Lambda[|E|].
\end{align}
where the function $\Lambda(s)$ is an anti-derivative of $\lambda(s)$, and $\mathrm{Per}(E)$ denotes the perimeter of $E$. For the sets $E$ and $F$ in $\R^n$, we use the pseudo-distance defined by
\begin{align*}
\tilde{d}(F,E) := \left({\int_{E \triangle F} d(x,\partial E) dx}\right)^{\frac{1}{2}}, \quad E \triangle F := (E\setminus F) \cup (F\setminus E).
\end{align*}

We consider minimizing movements for the restricted equation \eqref{main_eq22} in a finite time interval $[0,T]$ with initial data \eqref{initial} with the admissible sets
\begin{align}\label{admissible}
A_M(E) := \{ F \in S_{r_0,R_0} \mid d_H(\partial(F \cap E),\partial E) \leq Mh \} \hbox{ for } E \subset \R^n, 
\end{align}
with  
\begin{align}\label{restriction}
r_0< r_1 =r_1(\rho,a) = \rho(a^2+2a)^{1/2} \hbox{ and } R_0 > R_1
\end{align}
where $r_1$ is given in \eqref{eqn_pre_star} and $R_1= R_1(T)$ in \eqref{eqn_bound00}.
Recall that $\rho$ is given in Definition~\ref{rho-ref} and $a$ is given in Lemma~\ref{Brho}. The dependence of $R_1$ on $T$ is the reason why we restrict the discussion in this and next section to the finite time interval. For simplicity we will omit the time dependence in $R_1$ and thus in $R_0$.  
 
\begin{DEF}
\label{dis-flow}
For $h > 0$, $T_{h,M}$ is defined by
\begin{align*}
T_{h,M}(E) \in \argmin_{F \in A_M(E)} I_{h}(F;E), \quad I_{h}(F;E) := J(F) + \frac{1}{h}\tilde{d}^2(F,E),
\end{align*}
The existence of a minimizer, $T_{h,M}(E)$ follows from Lemma \ref{lem-cpt}, \ref{lem-cpt2} and \ref{beer}.

\textit{The constrained minimizing movement} $E^{h,M}_t$ of $J$ for $t \in [0, T]$ with initial set $E_0$ can be defined by
\begin{align*}
E^{h,M}_t := T^{[t/h]}_{h,M}(E_0).
\end{align*}
Here, $T^{m}$ for $m\in \mathbb{N}$ is the $m$-th functional power. 
\end{DEF}

\begin{DEF}
\label{def-ren}
A function $w_M:=\chi_{E_t} - \chi_{E_t^C}$ is \textit{a restricted flat flow} of \eqref{main_eq22} and \eqref{initial} if $E_0 = \Omega_0$ and there exists a sequence $h_k \rightarrow 0$ such that
\begin{align*}
d_H(E_t, E_{t}^{h_k,M}) \rightarrow 0
\end{align*}
locally uniformly in time as $k$ goes to infinity.
\end{DEF}

To show the existence of a restricted flat flow, let us show compactness property of the constrained minimizing movements.

\begin{lem} 
\label{GF_inequality}	
The constrained minimizing movement $E^{h,M}_t$ in Definition \ref{dis-flow} satisfies the following inequality for $0<t<s\leq T$:
\begin{align}\label{first}
\tilde{d}^2(E^{h,M}_{s}, E^{h,M}_{t}) \lesssim_{r,R} (s-t)(J(E^{h,M}_{t}) - J(E^{h,M}_{s}))
\end{align}
and, as a consequence,
\begin{align}\label{second}
d_H(E^{h,M}_{s}, E^{h,M}_{t})\lesssim_{r,R}  |t - s |^{\frac{1}{n+1}}.
\end{align}
\end{lem}

\begin{proof}
 We will use the triangle-like inequality (see e.g.  Lemma 17, \cite{Feldman:2014hb}):
\begin{align}
\label{ps-dis}
\frac{\tilde{d}^2(F_{k+1},F_1)}{k} \lesssim_{r,R} \sum_{j=1}^{k}\tilde{d}^2(F_{j+1},F_{j}) \hbox { for } F_1, ..., F_{k+1} \in S_{r,R}.
\end{align}

Suppose that $t \in [Kh,(K+1)h)$ and $s \in [(K+L)h,(K+L+1)h)$ for some $K$ and $L>0$.
By the construction of $E^{h,M}_t$ in Definition \ref{dis-flow} for $k \in N$,
\begin{align*}
J(E^{h,M}_{(k-1)h}) - J(E^{h,M}_{kh}) \geq \frac{1}{h}\tilde{d}^2(E^{h,M}_{kh},E^{h,M}_{(k-1)h}).
\end{align*}
By summing both sides from $k=K+1$ to $k= K + L$,
\begin{align*}
J(E^{h,M}_{Kh}) - J(E^{h,M}_{(K+L)h}) &\geq \sum_{k=K+1}^{K+L}\frac{1}{h}\tilde{d}^2(E^{h,M}_{kh},E^{h,M}_{(k-1)h}),\\
&\gtrsim_{r,R} \frac{1}{Lh}\tilde{d}^2(E^{h,M}_{(K+L)h}, E^{h,M}_{Kh}),
\end{align*} 
where the last inequality follows from \eqref{ps-dis}. \eqref{second} follows from Lemma~\ref{lem-cpt2}.
\end{proof}

One can apply Lemma \ref{GF_inequality} and compactness of star-shaped sets (Lemma \ref{lem-cpt}, \ref{lem-cpt2} and \ref{beer}) to obtain the following:

\begin{thm}
\label{cor-exi-ren}
There exists at least one \textit{restricted flat flow} $w_M$ of \eqref{main_eq22} and \eqref{initial} in the sense of Definition~\ref{def-ren}.
\end{thm}

\subsection{Barrier Property under Star-shapedness}

Next we establish a ``restricted barrier property" for a restricted flat flow with respect to a classical subsolution and supersolution of \eqref{main_eq22} with $\eta(t) = \lambda[|\Omega_t(w_M)|]$.  The proof of this proposition is rather technical and follows that of \cite{Grunewald:2011cb}: see Appendix \ref{app_pf1}. In a different setting, similar results are shown in \cite{CMNP19} and \cite{CMP15}.

\begin{prop}
\label{cor_bar}
(Restricted barrier property)
Let $w_M$ be a restricted flat flow of \eqref{main_eq22} with the admissible set constraint parameters $r_0$ and $R_0$ satisfying \eqref{restriction}. For any $r>r_0$ and $R<R_0$, suppose that there exists a test function $\phi$ on $Q_T$ such that $\phi$ is a classical subsolution in $Q_T$ of \eqref{main_eq22} with $\eta(t) = \lambda[|\Omega_t(w_M)|]$, $|D\phi|\neq 0$ on $\partial\Omega_t(\phi)$ and  $\Omega_t(\phi) \in S_{r,R}$ in $[0,T]$.
If $\Omega_0(\phi) \subset \subset \Omega_0(w_M)$,
then 
$$
\Omega_t(\phi) \subset \subset \Omega_t(w_M)\quad \hbox{ for all  }t \in [0,T].
$$

Similarly, suppose that there exists a test function $\psi$ on $Q_T$ such that $\psi$ is a classical supersolution in $Q_T$ of \eqref{main_eq22} with $\eta(t) = \lambda[|\Omega_t(w_M )|]$, $|D\psi|\neq 0$ on $\partial\Omega_t(\psi)$ and $\Omega_t(\psi) \in S_{r,R}$ in $[0,T]$. If $\Omega_0(w_M ) \subset \subset \Omega_0(\psi)$, then $$\Omega_t(w_M ) \subset \subset \Omega_t(\psi) \quad \hbox{ for all  }t \in [0,T].$$
\end{prop}

In the proof of Proposition \ref{cor_bar}, we only use the properties of the classical solution $\phi$ in small neighborhood of $(x_0,t_0)$, thus we can deduce the following localized barrier property of the flat flow.

\begin{cor}
\label{cor_bar2}
Let $w_M$ be a restricted flat flow of \eqref{main_eq22} with the admissible set constraint parameter $r_0$ and $R_0$ satisfying \eqref{restriction}. If there exists a test function $\phi$ on $Q_T$ such that $\phi$ touches $w$ from below at $(x_0,t_0)$, $|x_0| < R_0$, $|D\phi|(x_0,t_0)\neq 0$ and 
$- x_0 \cdot \frac{D\phi}{|D\phi|}(x_0,t_0)  > r_0$.
then
$$\dfrac{\phi_t}{|D\phi|}(x_0,t_0) \geq \max \left\{  \nabla \cdot \left(\dfrac{D\phi}{|D\phi|}\right)(x_0,t_0)+ \eta(t_0), -M \right\}.$$

Similarly, if there exists a test function $\psi$ on $Q_T$ such that $\psi$ touches $w$ from above at $(x_0,t_0)$, $|x_0| < R_0$, $|D\psi|(x_0,t_0)\neq 0$ and 
$- x_0 \cdot \frac{D\psi}{|D\psi|}(x_0,t_0)  > r_0$
then
$$\dfrac{\psi_t}{|D\psi|}(x_0,t_0) \leq  \max \left\{ \nabla \cdot \left(\dfrac{D\psi}{|D\psi|}\right)(x_0,t_0)+ \eta(t_0), -M\right\}.$$
\end{cor}

\section{Existence of the Flow : Coincidence between Flat Flows and Viscosity Solutions}
\label{sec_exi}

Our goal in this section is to show the existence of a viscosity solution for \eqref{main eq1}.

\medskip

Let us give a brief summary of this section. We will show that a restricted flat flow coincides with the corresponding viscosity solution as long as the viscosity solution is star-shaped (Proposition~\ref{comp-en-vis}).  Ensuring this star-shaped property for the viscosity solution is the last step leading to the coincidence result (Proposition~\ref{short time}): this is where we need the lower bound $M$ on the velocity of the flow imposed by \eqref{main_eq22}. After we show the coincidence between a restricted flat flow and the corresponding viscosity solution, we address removing the bound $M$ to obtain our desired result (Theorem~\ref{coro}).

\subsection{Coincidence for the restricted problem}

In this section our goal is to show coincidence between flow flows and viscosity solutions for the restricted problem \eqref{main_eq22}.  To this end we first show a comparison result between a flat flow and the corresponding viscosity solution of \eqref{main_eq22}. We use the doubling argument in \cite{Crandall:1992kn} and \cite{Kim:2005fi} which preserves the star-shaped geometry of the level sets of the solutions.

\begin{prop}
\label{comp-en-vis}
Let $w_M$ be a restricted flat flow of \eqref{main_eq22} with the admissible set constraint parameter $r_0$ and $R_0$ satisfying \eqref{restriction}.
Suppose that there exists a viscosity subsolution $u : Q_T \rightarrow \R$ of \eqref{main_eq22} with $\eta(t) = \lambda[|\Omega_t(w_M)|]$ such that $\Omega_t(u)$ is in $S_{r,R}$ for all $t \in [0,T]$ for some $r>r_0$ and $R<R_0$.
If $
\Omega_0(u)  \subset \subset\Omega_0(w_M),
$ then 
$$
\Omega_t(u) \subset\subset \Omega_t(w_M)\quad \quad\hbox{ for all } t \in [0,T].
$$ 

Similarly, suppose that there exists a viscosity supersolution $u : Q_T \rightarrow \R$ of \eqref{main_eq22} with $\eta(t) = \lambda[|\Omega_t(w_M)|]$  such that $\Omega_t(u)$ is in $S_{r,R}$ for all $t \in [0,T]$ for some $r>r_0$ and $R<R_0$. 
If $
\Omega_0(w_M) \subset\subset \Omega_0(u),
$ then 
$$
\Omega_t(w_M) \subset\subset \Omega_t(u)\quad\hbox{ for all } t \in [0,T].
$$ 
\end{prop}

\begin{proof}

The proof follows the outline of  \cite{Kim:2005fi}, where the comparison principle is shown for a nonlocal mean-curvature flow.

\medskip

 For $c,\delta>0$, let us consider
$$
Z(x,t) := \sup_{|z| \leq c-\delta t} u(x+z,t) \hbox{ and } 0\leq t\leq \frac{c}{\delta},
$$
where $c$ is chosen sufficiently small so that  $\Omega_0(Z) \subset\subset \Omega_0(w_M)$.  Due to Lemma~\ref{sup convolution_gen}, the function $Z$ is a viscosity subsolution of \begin{align*}
u_t =  F(Du,D^2u,t) - \delta |Du|.
\end{align*} 
We will show Proposition 6.2 by showing that for any $\delta>0$ and $0\leq t\leq c/\delta$ we have 
\begin{align}\label{claim}
\Omega_t(Z) \subset \subset \Omega_t(w_M).
\end{align}

\medskip

 Note that for any $z \in \R^n$ such that $|z| \leq c$, the interior cone $IC(x,r)$ given in \eqref{eqn_ic} satisfies $IC(x+z,r-c) \subset IC(x,r) + z$ (See Lemma~\ref{lem:iic}). Thus, by the equivalence relation in Lemma~\ref{star}, $\Omega_t(u) \in S_{r,R}$ implies that $\Omega_t(u) + z \in S_{r - c,R+c}$ for all $|z|  \leq c$ and thus
\begin{align*}
\Omega_t(Z) = \bigcup_{|z| \leq c-\delta t} [\Omega_t(u) +z] \in S_{r-c,R+c}.
\end{align*}
Thus, $\Omega_t(Z) \in S_{r_0+c,R_0-c}$ for $0< c \leq \min\left\{\frac{r-r_0}{2}, \frac{R_0-R}{2}\right\}$.
 
\medskip
 
Suppose \eqref{claim} is false, then we have
$$
t_0:= \sup\{ t: \Omega_s(Z) \subset\subset \Omega_s(u) \hbox{ for } 0\leq s\leq t\}  \in (0, c/\delta).
$$
Due to Lemma~\ref{cor_holder} and Lemma~\ref{GF_inequality}, both sets $\partial\Omega_t(Z)$ and $\partial\Omega_t(w_M)$ evolve continuously in time.  Hence, $\partial \Omega(Z)$ touches $\partial \Omega(w_M)$ from inside for the first time at $t=t_0\in (0,\frac{c}{\delta})$. 

\medskip

For $\e \in (0,\frac{\delta}{2n})$, let us define $\tilde{Z} := \chi_{\bar{\Omega}(Z)} $ and $\tilde{W} := \chi_{\Omega(w_M)}$ and 
$$\Phi_\e(x,y,t) := \tilde{Z}(x,t) - \tilde{W}(y,t) - \frac{|x-y|^4}{4\e} - \frac{\e}{2(t_0 - t)}.$$
Let $d_0$ be distance between $\partial\Omega_0(Z)$ and $\partial\Omega_0(w_M)$. Since $\tilde{Z} -  \tilde{W}$ is bounded, we can choose a sufficiently small $\e << d_0^4$ such that $\Phi(x,y,0) < 0$ for all $x$ and $y$.  

\medskip

Since the function $\tilde{Z} -  \tilde{W}$ is upper semicontinuous and bounded above by zero for all $t<t_0$, the function $\Phi_\e(x,y,t)$ has a local maximum at $(x_\e,y_\e,t_\e)$ in $\R^n \times [0,t_0)$ for any $\e$. By H\"{o}lder continuity of $\partial\Omega(Z)$ and $\partial\Omega(w_M)$ from Lemma~\ref{cor_holder} and Lemma~\ref{GF_inequality}, there exists $x_1 \in \partial\Omega_{t_0 - \e}(\tilde{Z})$ and $y_1 \in \partial\Omega_{t_0 - \e}(\tilde{W})$ such that $|x_1 - y_1 | \leq K\e^{\frac{1}{2}}$ where $K$ depends on H\"{o}lder constants of $\partial\Omega(Z)$ and $\partial\Omega(w_M)$. For $\e << K^{-4}$, it holds that $\Phi(x_\e,y_\e,t_\e) > \Phi(x_1,y_1,t_0 - \e) > \frac{1}{3}$, and thus $t_\e \in (0,t_0)$. Also, $\Phi(x_\e,y_\e,t_\e)$ is uniformly bounded from below in $\e$, and thus it holds that $|x_\e - y_\e| = O(\e^{\frac{1}{4}})$.

\medskip

Moreover, since $\tilde{Z} - \tilde{W} > \Phi>\frac{1}{3}$ at $(x_\e,y_\e,t_\e)$, we conclude that $x_\e \in \Omega_{t_\e}(\tilde{Z})$, $y_\e \in \Omega_{t_\e}(\tilde{W})^C$. As $t_0$ is the first touching point and $t_\e < t_0$, it holds that $|x_\e - y_\e| >0$. On the other hand, $\tilde{Z}(x,t_\e) - \tilde{W}(y,t_\e) = 1$ for all $(x,y) \in \Omega_{t_\e}(\tilde{Z}) \times \Omega_{t_\e}(\tilde{W})^C$, and thus $(x_\e, y_\e)$ is a maximizer of the third term $- \frac{|x-y|^4}{4\e}$ in $\Omega_{t_\e}(\tilde{Z}) \times \Omega_{t_\e}(\tilde{W})^C$. We conclude that $x_\e$ and $y_\e$ are on $\partial \Omega_{t_\e}(\tilde{Z})$ and $ \partial \Omega_{t_\e}(\tilde{W})$, respectively. 

\medskip

Then, as equation (2.9) in \cite{Kim:2005fi}, there exist quadratic test functions $\phi^\e(x,t)$ and $\psi^\e(x,t)$ such that
\begin{align}
\label{comp-1}
\begin{cases}
\phi^\e(x,t) := [a_\e ( t-t_\e ) + p_\e \cdot (x-x_\e) + \frac{1}{2} (x-x_\e)^T X_\e (x-x_\e) ]_+ \geq \tilde{Z}(x,t) &\qquad \hbox{ in } N_1^\e, \\
\psi^\e(y,t) := [b_\e ( t-t_\e ) + q_\e \cdot (y-y_\e) + \frac{1}{2} (y-y_\e)^T Y_\e (y-y_\e) ]_+ \leq \tilde{W}(y,t) &\qquad \hbox{ in } N_2^\e, 
\end{cases}
\end{align}
where constants $a_\e,b_\e \in \R$, $p_\e,q_\e = \frac{x_\e-y_\e}{\e} +O(\e^2) \in \R^n\setminus\{0\}$, $X_\e, Y_\e \in S^{n\times n}$, neighborhoods $N_1^\e$ of $(x_\e,t_\e)$ and $N_2^\e$ of $(y_\e,t_\e)$ satisfying the inequalities:
 \begin{align}
\label{comp-2}
\begin{cases}
a_\e - b_\e &\geq 0,\\
X_\e - Y_\e &\leq \e |p_\e| I, \\
||p_\e| - |q_\e|| &\leq \e^2 \min\{1,|p_\e|^2\},\\
|p_\e - q_\e| &\leq \e^2 \min\{1,|p_\e|^2\}.
\end{cases}
\end{align}

For simplicity we carry out the computations in steps 4-5. for the case $M=\infty$.  Since $\tilde{Z}$ is a viscosity solution and $\phi^\e$ touches $\tilde{Z}$ from above at $(x_\e,t_\e)$, it holds that
$$ \frac{a_\e}{|p_\e|} = \dfrac{\phi^\e_t}{|D\phi^\e|}(x_\e,t_\e) \leq  \nabla \cdot \left(\dfrac{D\phi^\e}{|D\phi^\e|}\right)(x_\e,t_\e) + \eta(t_\e) - \delta = \frac{1}{|p_\e|} \left(\mathrm{trace}(X_\e) - \frac{p_\e^T X_\e p_\e}{|p_\e|^2}\right)+ \eta(t_\e) - \delta.$$ 
By inequalities \eqref{comp-2} and the ellipticity of the operator, $\mathrm{trace}(X) - \frac{p^T X p}{|p|^2},$ it can be seen that
\begin{align*}
\frac{b_\e}{|p_\e|} \leq  \frac{a_\e}{|p_\e|} &\leq  \frac{1}{|p_\e|} \left(\mathrm{trace}(X_\e ) - \frac{p_\e^T X_\e p_\e}{|p_\e|^2}\right)+ \eta(t_\e) -\delta, \\
&\leq \frac{1}{|p_\e|} \left(\mathrm{trace}(Y_\e
) - \frac{p_\e^T Y_\e p_\e}{|p_\e|^2}\right) + \eta(t_\e) - \frac{\delta}{2}.
\end{align*}
Thus, by \eqref{comp-2}, for sufficiently small $\e>0$, it holds that
\begin{align}
\label{comp-3}
\frac{b_\e}{|q_\e|} \leq  \frac{1}{|q_\e|} \left(\mathrm{trace}(Y_\e) - \frac{q_\e^T Y_\e  q_\e}{|q_\e|^2}\right)+ \eta(t_\e)  - \frac{\delta}{4}.
\end{align}

Moreover, as $\Omega_t(\tilde{Z}) \in S_{r_0+c,R_0-c}$, $|x_\e| < R_0 - c$ and Lemma \ref{star2} implies that
$$x_\e \cdot \left(-\frac{p_\e}{|p_\e|}\right) \geq r_0+c.$$ 
There exists sufficiently small $\e_0$ such that for any $\e \in (0,\e_0)$, 
\begin{align}
\label{comp-4}
|y_\e| < R_0, \hbox{ and } y_\e \cdot \left(-\frac{q_\e}{|q_\e|}\right) > r_0.
\end{align}
This contradicts Corollary \ref{cor_bar2}. since $\psi^\e$ touches $\tilde{W}$ from below at $(y_\e,t_\e)$, but satisfies \eqref{comp-3} and \eqref{comp-4}.
\end{proof}

Next we will show that viscosity solutions $u$ of \eqref{main_eq22} has a short time star-shapedness property. 

\begin{DEF}
	\label{def_cs}
	 (Inf-Sup Convolutions with Space Scaling)
	For $\delta, \e>0$ and $M$ as in \eqref{main_eq22}, define ${\tilde{u}^S} , {\hat{u}^S}: \R^n \times \left[0,\frac{\delta}{5M}\right) \rightarrow \R$ as 
	 	\begin{align}
	\label{eqn_inf}
	{\tilde{u}^S}(x,t;\delta,\e) := \left. \inf \left\{  u\left(\frac{y}{1+\e},t\right) \,\right| \, \hbox{ } y \in \overline{B_{\delta\e-5M\e t}(x)} \right\},
	\end{align}
	and
	\begin{align}
	\label{eqn_sup}
	{\hat{u}^S}(x,t;\delta,\e) := \left. \sup \left\{  u\left(\frac{y}{1-\e},t\right) \,\right| \, \hbox{ } y \in \overline{B_{\delta\e-5M\e t}(x)} \right\},
	\end{align}
\end{DEF}

 Note that $\Omega_0(\hat{u}^S) \subset \Omega_0(u) \subset \Omega_0(\tilde{u}^S)$ due to \eqref{eqn_set_inf} and \eqref{eqn_set_sup}.

\begin{lem}
\label{sup convolution2}
Let u be \textit{a viscosity solution} of \eqref{main_eq22}. Suppose that $\Omega_0(u) \in S_{r}$ and $|\eta(t)| < M$ in $[0,T]$. Then, for any fixed $0<\delta<r$, there exists $\e_0 \in (0,1) $ such that ${\tilde{u}^S}$ (${\hat{u}^S}$) is a viscosity supersolution (subsolution) of \eqref{main_eq22} for all $0<\e < \e_0$. 
\end{lem}

\begin{proof}
We only prove for ${\tilde{u}^S}$, the proof for ${\hat{u}^S}$ is parallel. Let us define
$v(x,t) := u\left(\frac{x}{1+\e},t\right)$
in $\oQ$. Since $u$ is a viscosity solution of \eqref{main_eq22},
$v$ solves
$$
\frac{v_t}{ |Dv|} = (1+\e) \max \left[ (1+\e)  \nabla \cdot \left(\frac{Dv}{|Dv|}\right) + \eta(t),  - M\right]
$$
in the viscosity sense. Proceeding as in the proof of Lemma \ref{sup convolution_gen} one can  verify that ${\tilde{u}^S}$ satisfies
$$
\frac{{\tilde{u}^S}_t}{ |D{\tilde{u}^S}|} \geq (1+\e) \max \left[ (1+\e)  \nabla \cdot \left(\frac{D{\tilde{u}^S}}{|D{\tilde{u}^S}|}\right) + \eta(t),  - M\right]  + 5M \e.
$$
	
For simplicity, let us denote $H := - \nabla \cdot \left(\frac{D{\tilde{u}^S}}{|D{\tilde{u}^S}|}\right)$. Then, the right hand side is
$$(1+\e) \max \left[ - (1+\e)  H + \eta(t),  - M\right]  + 5M \e = \max \left[ - (1+\e)^2  H + (1+\e) \eta(t) + 5M \e,  - M + 4M\e \right].$$
	
First suppose that $-H + \eta \geq - M$. Since $|\eta(t)| < M$ and $\e<1$, it holds that 
\begin{align*}
- (1+\e)^2  H + (1+\e) \eta + 5M \e
&\geq - H + (\e^2 + 2\e)(- \eta - M) + (1+\e) \eta + 5M \e,\\
&\geq -H + \eta - (\e^2 + \e) \eta + (3 \e - \e^2) M \geq -H + \eta.
\end{align*}
If $-H + \eta < - M$, then $-M + 4M \e \geq -M = \max \left\{ -H +\eta , - M\right\}$, so we conclude that 
$$\frac{{\tilde{u}^S}_t}{ |D{\tilde{u}^S}|} \geq \max \left[   \nabla \cdot \left(\frac{D{\tilde{u}^S}}{|D{\tilde{u}^S}|}\right) + \eta(t),  - M\right].$$
\end{proof}

\begin{prop} \label{short time} (Short-time star-shapedness) Let $u$ be a viscosity solution of \eqref{main_eq22} and \eqref{initial}. Suppose that $\Omega_0(u) \in S_{r,R}$ and $|\eta(t)| < M$ for $0\leq t\leq \frac{r}{5M}$. Then, $\Omega_t(u) \in S_{r-5Mt,R+Mt}$ for $t \in \left[0, \frac{r}{5M}\right)$. 
\end{prop}
\begin{proof}

Applying Theorem \ref{comparison} for \eqref{main_eq22} to $u$ and $\tilde{u}^S$, we have
\begin{align*}
\Omega_t(u) \subset\subset \Omega_t({\tilde{u}^S})
\end{align*}
By Definition~\ref{def_cs} of $\tilde{u}^S$, it holds that
\begin{align}
\Omega_t(u) \subset\subset \bigcap_{|z| \leq {(\delta  - 5Mt)\e}} [(1+\e)\Omega_t(u) +z]
\end{align}
for all $0<\e < \e_0$, $0 < \delta < r$ and $0\leq t < \frac{\delta}{5M}$ where $\e_0$ is given in Lemma~\ref{sup convolution2}. Then, by \eqref{eqn_star} in Lemma~\ref{star}, we conclude that $\Omega_t(u) \in S_{r-5Mt}$.

\medskip

On the other hand, let us compare $u$ with a radial barrier $\phi$ defined by $$\phi :=  \chi_{B_{r(t)}} - \chi_{{B_{r(t)}}^C},$$ where $r:[0,T] \rightarrow \R$ is defined by $r(t) := R +Mt$. Since $\phi$ is a viscosity supersolution, comparison principle implies $\Omega_t(u) \subset B_{R+Mt}$. Thus, $\Omega_t(u) \in S_{r-5Mt,R+Mt}$ for $t \in \left[0, \frac{r}{5M}\right)$. 
\end{proof}

Now we are ready to prove our main theorem in this subsection.

\begin{thm}\label{coincidence}
Suppose that $\Omega_0$ has $\rr$. Let $r_0$ and $R_0$ satisfy \eqref{restriction}, and $\tilde{K}_{\infty}=\tilde{K}_{\infty}(r_0,R_0)$ be as in Lemma~\ref{cor_holder} . For $M>\tilde{K}_{\infty}$, consider a restricted flat flow $w_M$ of \eqref{main_eq22} and \eqref{initial} in the sense of Definition~\ref{def-ren}. If  $u$ is a unique viscosity solution of \eqref{main_eq22} and \eqref{initial} with $\eta(t) = \lambda[|\Omega_t(w_M)|]$, then $w_M=u$ in $\oQ$.
\end{thm}

\begin{proof}
The existence and short time uniqueness of $u$ for the above choice of $\eta(t)$ follows by Theorem~\ref{comparison} and Theorem~\ref{thm-uni-2}. 

Recall that $\Omega_0 \in S_{r_1,R_1}$ where $r_1$ and $R_1$ are given in \eqref{eqn_pre_star} and \eqref{eqn_bound00}.  Let us first show that $u = w_M$ in the small time interval $I=[0,t_0]$, where $t_0:=\min\left\{\frac{r_1-r_0}{10M}, \frac{R_0-R_1}{2M} \right\}$. As Corollary~\ref{pre-star2}, we can make $\Omega_0$ strictly smaller $\Omega_0^{\e,-}$ or bigger $\Omega_0^{\e,+}$ by dilation and can still make it stay in $S_{r_\e,R_\e}$ with $r_\e=r_1-O(\e) > r_0$ and $R_\e=R_1+O(\e) > R_0$, where $\e$ can be chosen arbitrarily small such that $r_\e - r_0 > \frac{r_1 - r_0}{2}$ and $R_0 - R_\e > \frac{R_0 - R_1}{2}$. Let us choose to make the domain strictly bigger, $\Omega_0^{\e,+}$, we can apply Proposition~\ref{short time} to ensure that the corresponding viscosity solution $u^\e$ of \eqref{main_eq22} satisfies, for some $r>r_0$ and $R<R_0$, 
$$
\Omega_t(u^\e)\in S_{r,R} \hbox{ for } t \in I.
$$  We can then apply Proposition~\ref{comp-en-vis} to $u^\e$ and $w_M$ to yield that 
\begin{align}\label{order00}
\Omega_t(w_M) \subset \Omega_t(u^\e)\hbox{ for }t\in I.
\end{align}
  Now to send $\e\to 0$, note that $\Omega_t(u^\e)$ satisfies  H{\"o}lder continuity, Corollary~\ref{holder0}. Thus along a sequence $\e=\e_n\to 0$,  $\Omega_t(u^\e)$ converges to a domain $\Omega_t\in S_{r,R}$ uniformly with respect to $d_H$ in the time interval $I$. Lemma~\ref{lem_sta} then yields that the corresponding level set function $u$ for $\Omega_t$ is the unique viscosity solution of $\eqref{main_eq22}$ with the initial data $u_0$. From \eqref{order00} we have 
$$
\Omega_t(w_M) \subset \Omega_t=\Omega_t (u)\hbox{ for } t\in I.
$$
Similarly,  using $\Omega_0^{\e,-}$ instead of $\Omega_0^{\e,+}$ we can conclude that $\Omega_t(u)\subset \Omega_t(w_M)$ and thus it follows that they are equal sets for the time interval $I$.

\medskip

3.  Once we know that $u=w_M$ in $I$, we know that $\eta(t)$ equals $\lambda[|\Omega_t(u)|]$ in $I$, and thus Theorem~\ref{pre-star} and Lemma~\ref{bound00} applies and now we know that $\Omega_t(u)\in S_{r_1,R_1}$ for $t\in I$. Now we can repeat the argument at  $t=t_0$ over the time interval $t_0 +I$, using the fact that $\Omega_{t_0}(u)\in S_{r_1,R_1}$. Now we can repeat above arguments to obtain that $w_M=u$ for all times.
\end{proof}

\subsection{Existence}
Let us define the notion of {\it flat flows} for our original problem.

\begin{DEF}
	\label{def-ene}
	A function $w:=\chi_{E_t} - \chi_{E_t^C}$ for $E_t \subset \R^n, t>0$ is \textit{a flat flow} of \eqref{main eq1} and \eqref{initial} if $E_0 = \Omega_0$ and there exists a sequence $M_k \rightarrow \infty$ such that
	\begin{align*}
	d_H(E_t,\Omega_{t}(w_{M_k})) \rightarrow 0
	\end{align*}
	locally uniformly in time as $k$ goes to infinity. Here, $w_{M_k}$ is a restricted flat flow with $M=M_k$ in the sense of Definition \ref{def-ren}.
\end{DEF}

Let us first show existence of the flat flow.

\begin{lem}
Suppose that $\Omega_0$ has $\rr$. There exists at least one \textit{flat flow} $w$ of \eqref{main eq1} and \eqref{initial}.
\end{lem}
\begin{proof}
Due to Corollary \ref{holder0} and Theorem~\ref{coincidence}, we have
$$d_H(\Omega_t(w_{M}), \Omega_s(w_{M})) \leq C|s-t|^{1/2}$$
where $C$ does not depend on ${M}$. Thus along a sequence $M_k\to\infty$ such that $\Omega_t(w_{M_k})$ converges with respect to $d_H$ to $\Omega_t$, locally uniformly in time. We conclude that $w:=\chi_{\Omega_t} -\chi_{\Omega_t^C}$ of \eqref{main eq1} is a flat flow in the sense of Definition \ref{def-ene}.
\end{proof}

Now, let us show existence and uniqueness of viscosity solution of our original equation \eqref{main eq1}.

\begin{thm}\label{coro}
Suppose that $\Omega_0$ has $\rr$. Let $w$ be a flat flow of \eqref{main eq1} and \eqref{initial} and let $u$ be the unique viscosity solution of \eqref{main eq2} and \eqref{initial} with  $\eta(t) = \lambda[|\Omega_t(w)|]$. Then $w=u$ in $\oQ$.  In other words, $w$ is the unique viscosity solution of \eqref{main eq1} and \eqref{initial}.
\end{thm}

\begin{proof}
By the construction of a flat flow, there exists a sequence $M_k \rightarrow \infty$ such that
\begin{align*}
d_H(\Omega_t(w),\Omega_{t}(w_{M_k})) \rightarrow 0
\end{align*}
uniformly in time as $k$ goes to infinity. Note that by Theorem \ref{coincidence}, any restricted flat flow $w_{M_k}$ is the unique viscosity solution of \eqref{main_eq22} with $\eta(t) = \lambda[|\Omega_t(w_{M_k})|]$. Then, Lemma~\ref{lem_sta} implies that $w$ is a viscosity solution of \eqref{main eq2} with $\eta(t) = \lambda[|\Omega_t(w)|]$. By the uniqueness of the viscosity solution of \eqref{main eq2} (Theorem \ref{comparison}), we conclude that $w=u$ in $\oQ$.

\medskip

Thus, $w$ is a viscosity solution of $\eqref{main eq1}$ with $\eta(t) = \lambda[|\Omega_t(w)|] = \lambda[|\Omega_t(u)|]$. From Theorem~\ref{cor-uni-2}, it is unique.
\end{proof}

We expect that the unconstrained minimizing movement scheme gives the parallel results in Theorem~\ref{coro}. In the subsequent work \cite[Proposition 3.4]{KK19}, we show that a viscosity solution of \eqref{main eq1} can be approximated by 
a minimizing movement scheme with an admissible set $S_{r_0, R_0}$ instead of $A_M$ in \eqref{admissible}.
 
\section{Regularity and Convergence}
\label{sec_reg}

In this final section we discuss the large time behavior and exponential convergence to equilibrium for solutions of \eqref{main1}, for which the corresponding energy is  for $\gamma<-\frac{1}{n}$
\begin{align}\label{energy_specific}
J(E):= 
\begin{cases}
\mathrm{Per}(E)  - \dfrac{|E|^{\gamma+1}}{\gamma+1} & \hbox{ if } \gamma\neq -1
\\
\\ \mathrm{Per}(E)  - \ln |E| & \hbox{ if } \gamma =  -1
\end{cases}
\end{align}
Let us point out that here the dissipation of energy is crucial to guarantee that any subsequential limit of the evolving set over time variable must correspond to the unique equilibrium solution. 

\medskip

Let $u$ be the unique viscosity solution of \eqref{main eq1} and \eqref{initial} obtained from Theorem~\ref{coro}. First, we show uniform boundedness for sets $\Omega_t(u)$ for all times.

\begin{lem}
\label{mmin}
There exists $R>0$ such that $\Omega_t(u) \subset B_R$ for all $t>0$.
\end{lem}
\begin{proof}
For simplicity we denote $\Omega_t(u)$ by $\Omega_t$. Recall that the energy $J$ decreases through the flow,
\begin{align*}
J(\Omega_0) &\geq J(\Omega_t) = \mathrm{Per}(\Omega_t) -  \Lambda[|\Omega_t|]
\end{align*}
By the isoperimetric inequality, it holds that
\begin{align*}
\mathrm{Per}(\Omega_t) -  \Lambda[|\Omega_t|]
 &\geq n w_n^{\frac{1}{n}} |\Omega_t|^{\frac{n-1}{n}}  -  \Lambda[|\Omega_t|]  =  |\Omega_t|^{\frac{n-1}{n}}\left(n w_n^{\frac{1}{n}} - \frac{\Lambda[|\Omega_t|]}{|\Omega_t|^{\frac{n-1}{n}}}\right).
\end{align*}
where $w_n$ is the volume of $B_1$. Since $\lambda(s) = s^{\gamma}$ with $\gamma<-1/n$, we have 
\begin{align*}
\lim_{s \rightarrow \infty}\dfrac{\Lambda[s]}{s^{\frac{n-1}{n}}}  = \dfrac{n}{n-1} \lim_{s \rightarrow \infty}  s^{\frac{1}{n}}\lambda[s]  = 0
\end{align*}
which yields that $|\Omega_t|$ is uniformly bounded. As $\Omega_t$ has $\rr$, we conclude.
\end{proof}

Next, we improve the regularity of $\Omega_t(u)$. Due to the fact that the support of $u$ is uniformly bounded (See Lemma~\ref{mmin}) and is in $S_r$ for all times, it follows that there exists $s_0, L_0>0$ such that  for any point $x\in \Omega_{t}(u)$ and $(t-s_0^2)_+ \leq s\leq t$, $\partial\Omega_s (u)$ can be represented as a Lipschitz graph in $ B_{s_0}(x)$ with its Lipschitz constant less than $L_0$. This fact, the a priori estimates obtained in Lemma~\ref{regularity} in the appendix, and the regularization procedure given in Theorem 3.4 of \cite{EH} yields the following short-time result.  

\begin{lem}\label{bound}
Suppose that $\Omega_0$ has $\rr$. For $t\geq 1$,  $\partial\Omega_t(u)$  is uniformly $C^{1,1}$. More precisely there exists $s_0, L_0>0$  such that for any point $x\in \Omega_{t}(u)$ and for $(t-s_0^2)_+\leq s\leq t$, $\partial\Omega_s(u)$ can be represented as a $C^{1,1}$ graph in $ B_{s_0}(x)$ with its $C^{1,1}$ norm less than $L_0$.
 \end{lem}

Next we show the convergence of $\Omega_t(u)$ in terms of the Hausdorff distance. This is due to the fact that the ball is the unique critical point of the perimeter energy among the class of $C^{1,1}$, star-shaped sets with given volume.

\begin{lem}
\label{uni-min}  $J(E)$ given in \eqref{energy_specific} has a unique minimizer $B_{r_{\infty}}$ among sets in $S_{\rho}$, up to translation. 
\end{lem}

\begin{proof} 
	
By the usual re-arrangement argument one can show that the minimizer is a ball (See for instance \cite[Theorem 14.1]{Maggi:2012tg}). By differentiating the energy $J(B_r)$ with respect to radius $r$ and recalling \eqref{main1}, we have
\begin{align*}
\frac{dJ(B_r)}{dr}
&= n(n-1) w_n r^{n-2} - n w_n r^{n-1} \lambda[|B_r|] =w_n r^{n-2} ( n-1 - r \lambda[|B_r|]),
\end{align*}
which is zero if and only if  $r \lambda[|B_r|]=n-1$.
Note that such $r$ is unique since $r\lambda[|B_r|] = w_n^\gamma r^{n\gamma+1}$ is a decreasing function on $r$. Let's denote this by
\begin{align}
\label{eqn_rinf}
r_\infty:= \left( \frac{n-1}{w_n^\gamma} \right)^{\frac{1}{n\gamma+1}}
\end{align}
As we have $\lambda[|B_{\rho}|] > \frac{n-1}{\rho}$ due to Assumption A, it follows that $r_\infty > \rho$.

\medskip

As $\frac{dJ(B_r)}{dr}$ is positive for $r>r_\infty$ and negative for $r<r_\infty$, we conclude that $r_\infty$ is a minimizer of $J(B_r)$.
\end{proof}

We show the uniform convergence of $\Omega_t(u)$ when $\Omega_0$ has $\rr$. Note that without the assumption on the initial set, the topology of $\Omega_t(u)$ may not be preserved as described in Figure~\ref{fig_non_star}. Thus it is unclear whether the flow converges to a ball.

\begin{thm}\label{first}
Suppose that $\Omega_0$ has $\rr$. Then, the set $\Omega_t(u)$ uniformly converges to a ball as $t \to \infty$, modulo translation. More precisely 
$$
\inf\{ d_H(\Omega_t(u), B_{r_\infty}(x)): x\in \overline{B_{\rho}(0)}\} \to 0 \hbox{ as } t\to\infty,
$$
where $r_\infty$ is given in \eqref{eqn_rinf}.
\end{thm}

\begin{proof}
For simplicity we denote $\Omega_t(u)$ by $\Omega_t$. Recall $R>0$ from Lemma~\ref{mmin}.
Define $U_n: [0,\infty) \to S_{r,R}$ by $U_n(t):= \Omega_{t+n}$.  Due to Corollary~\ref{holder0} the maps $U_n$ are a sequence of equicontinuous maps into $S_{r,R}$ in Hausdorff topology, and thus along a subsequence $U_n$ locally uniformly converges to $U_{\infty}: [0,\infty) \to S_{r,R}$.  From the standard stability theory of the viscosity solutions, it follows that $\chi_{U_{\infty}} - \chi_{U_{\infty}^C}$ is a viscosity solution of \eqref{main eq1}. Let us recall that the energy $J(\Omega_t)$ is monotone decreasing in time and $J : S_{r,R} \to \R$ is bounded from below. Thus, we have
\begin{align}
\label{01}
\lim_{n \to \infty} J(\Omega_n) = \lim_{n \to \infty} J(U_n(t)) = \lim_{n \to \infty} J(U_n(s)) \hbox{ for any } t,s >0.
\end{align}
Also recall from Lemma~\ref{GF_inequality} that we have 
\begin{align}
\label{00}
\tilde{d}^2(\Omega_t, \Omega_s) \lesssim_{r,R} |t-s| (J(\Omega_t) - J(\Omega_s)) \hbox{ for any } s > t >0.
\end{align}

\medskip

Combining \eqref{01} and \eqref{00} with the uniform convergence of $U_n$ in Hausdorff distance yields that for all $t,s>0$
$$
\tilde{d}^2(U_\infty(t), U_\infty(s)) \lesssim_{r,R} |t-s| \left(\lim_{n \to \infty} J(U_n(t)) - \lim_{n \to \infty} J(U_n(s))\right)=0,
$$
and thus $U_\infty$ is independent of time, and $\chi_{U_\infty}$ solves the prescribed mean curvature problem 
$$
H = \lambda(|E|).
$$

The only viscosity solution of the above problem, among sets in $S_{r,R}$ with $C^{1,1}$ boundaries, is smooth due to the fact that the corresponding level set PDE in the graph setting is uniformly elliptic with Lipschitz coefficients (See also Lemma~\ref{regularity}). From the Soap Bubble Theorem by Alexandrov in \cite{Ale58} and \cite{Ale62}, it follows that the only possible solution is radial, and so we conclude that $U_\infty = B_{r_\infty}$.
\end{proof}

To obtain exponential convergence, it is necessary to observe further regularity properties of $\Omega_t(u)$.  To this end, note that the following holds as a consequence of Lemma~\ref{bound} and Theorem~\ref{first}:

\begin{lem}\label{holder}
Suppose that $\Omega_0$ has $\rr$. For any $\e>0$ and $0<\alpha<1$ there exists $T$ and $C>0$ such that for any $t>T$ we have the following:
\begin{itemize}
\item[(a)] There exists $x_t\in B_{\rho}(0)$ such that for all $x \in \partial \Omega_t(u)$, the outward unit normal $\nu_x$ at $\partial\Omega_t(u)$ satisfies 
\begin{align}\label{first_bound}
\left|\nu_x - \frac{(x-x_t)}{|x-x_t|} \right| \leq 2C\e \hbox{ on } x\in \partial\Omega_t (u).
\end{align}
\item[(b)]
For all $x,y \in \partial \Omega_t(u)$, 
\begin{align}\label{second_bound}
\left| \left( \nu_x - \frac{(x-x_t)}{|x-x_t|} \right) - \left( \nu_y - \frac{(y-x_t)}{|y-x_t|} \right) \right| \leq C\e^{1-\alpha} |x-y|^{\alpha}.
\end{align}
\end{itemize}
\end{lem}

\begin{proof}

Choose a sufficiently small $\e$ depending on $\rho$. By Theorem ~\ref{first}, we can find $T>1$ such that for any $t>T$ there exists $x_t\in B_{\rho}(0)$ such that  
\begin{align}\label{confinement}
B_{r_\infty-\e^2}(x_t)\subset \Omega_t \subset B_{r_\infty+\e^2}(x_t).
\end{align}

Due to Lemma~\ref{bound}, the outward normal vector $\nu_x$ at $x\in \partial\Omega_t$  satisfies 
\begin{align}\label{zeroth_bound}
|\nu_x-\nu_y| \leq C_0|x-y|,
\end{align} 
where $C_0$ is independent of $t>1$. This means that if \eqref{first_bound} fails at $x_0\in\partial\Omega_t$ with sufficiently large $C$, let's say $C> C_0+2$, then $\nu_x$ stays different from  $\frac{(x_0-x_t)}{|x_0-x_t|}$ by at least $2\e$ in $\e$-neighborhood of $x_0$. As a result the boundary part $\partial\Omega_t \cap \partial B_{\e}(x_0)$ intersects the complement of the $2\e^2$-neighborhood of the tangential hyperplane of $B_r(x_t)$ at $x_0$, where $r= |x_0-x_t|$. This violates \eqref{confinement}, and thus we conclude that  \eqref{first_bound} holds.

\medskip

 To show (b), recall that due to \eqref{zeroth_bound} we have 
$$
 \dfrac{|\nu_{x} - \nu_{y}|}{|x-y|^{\alpha}} \leq Cd^{1-\alpha} \quad\hbox{ if } |x-y| \leq d \hbox{ for } 0< \alpha<1.
$$
On the other hand the same quantity is bounded by $\frac{2C\e}{d^{\alpha}}$ if $|x-y|\geq d$ due to \eqref{first_bound}. Hence choosing $d = \e$ we arrive at \eqref{second_bound}. 
\end{proof}

\medskip

Lemma~\ref{holder} states that after a finite time $\Omega_t$ gets arbitrarily close to a ball  in $C^{1,\alpha}$ norm. For the volume preserving mean curvature flow, \cite{Escher:1998} proves that when the initial domain is close to a ball in the sense of Lemma~\ref{holder} with sufficiently small $\e$, it converges to a unique round ball exponentially fast. Their analysis can be also applied to our problem with minor modifications:

\begin{thm}\label{exponential}
Suppose that $\Omega_0$ has $\rr$. The set $\Omega_t(u)$ exponentially converges to a unique ball of radius $r_\infty$ whose center depends only on the initial set $\Omega_0$, as $t\to\infty$.
\end{thm}

\begin{proof}
Parallel (center-manifold analysis) argument as in \cite{Escher:1998}, posed in the same function space, applies here since the difference between our problem and the volume preserving flow  lies in the Lagrange multiplier $\lambda(t)$, which is a lower order term compared to the mean curvature term.
\end{proof}

\noindent{\bf Acknowledgements.} We thank Will Feldman for helpful comments on the manuscript.

\appendix

\section{Proof of Proposition~\ref{cor_bar}}
\label{app_pf1}

\begin{proof}[Proof of Proposition~\ref{cor_bar}]
	
	1. We will prove the case $w_M <\psi$ at $t=0$, parallel proof holds for the other case.

	2. First, let us assume that $\Omega_t(\phi)$ touches $\Omega_t(w_M )$ from inside for
	the first time at $t=t_0$ at $x_0\in\Omega_{t_0}(w_M)$. Our goal is to make a perturbation of $\Omega_t(w_M)$ using $\Omega_t(\phi)$, which leads to a contradiction with the gradient flow property of $w_M$. To this end,
	let $\tilde{\phi}$ be a parallel translation of $\phi$ in the direction of normal vector at $x_0$, $\vec{n}_{x_0}$, so that $\Omega_{t_0}(\tilde{\phi})$ has nonempty intersection with the complement of $\Omega_t(w_M)$:
\begin{align}
\label{prop_eqn_energy_1}
 \tilde{\phi}(x,t) := \phi\left(x - \delta\left(e+(t-t_0)\right)\vec{n}_{x_0},t\right).
\end{align}
	Here, $e>0$ will be chosen in next step.
	Then, $U_t:=\Omega_{t}(\tilde{\phi}) \setminus \Omega_{t}(w_M)$ is nonempty at $t_0$ and we have 
	\begin{align}
	\label{energy-bar21}
	\dfrac{\tilde{\phi}_t}{|D\tilde{\phi}|}(x_0,t_0) \leq  \max \left\{ \nabla \cdot \left(\dfrac{D\tilde{\phi}}{|D\tilde{\phi}|}\right)(x_0,t_0) + \eta(t_0), - M \right\} - \delta.
	\end{align}
	
	\medskip
	
	3. Let us first assume that
	\begin{align}
	\label{energy-bar22}
	\dfrac{\tilde{\phi}_t}{|D\tilde{\phi}|}(x_0,t_0) \leq   \nabla \cdot \left(\dfrac{D\tilde{\phi}}{|D\tilde{\phi}|}\right)(x_0,t_0) + \eta(t_0) - \delta.
	\end{align}
	The other case will be discussed in step 4.
	
	For any $\e \in (0,\frac{\delta}{8 + 4C})$ where $C$ is defined in \eqref{eqn_energy_bar}, there exists sufficiently small $e \in (0, \frac{r_1 - r_0}{2}) $ such that (a) $ e \leq d_H(\Omega_t(\phi), \Omega_t(w_M))$ in $[t_0 - 4e, t_0 - 2e]$, (b) $|U_t| < \e$ in $[t_0 - 4e,t_0]$, and (c)
	\begin{align}
	\label{energy-bar5}
	\dfrac{\tilde{\phi}_t}{|D\tilde{\phi}|}(x,t) \leq  \nabla \cdot \left(\dfrac{D\tilde{\phi}}{|D\tilde{\phi}|} \right)(x,t)  + \eta(t) - \frac{\delta}{4}
	\hbox{ and }
	\left| \dfrac{\tilde{\phi}_t}{|D\tilde{\phi}|}(x,t) - \dfrac{\tilde{\phi}_t}{|D\tilde{\phi}|}(x,t_0) \right| < \frac{\e}{2}
	\end{align}
	in $\mathcal{N}_\e \times [t_0 - 4e,t_0]$
	where $\mathcal{N}_\e := \{ x : d(x,U_s) < \e \hbox{ for all } t_0 - 4e\leq s\leq t_0 \}.$
	
	\medskip

	Note that (a) implies $\Omega_t(\tilde{\phi}) \subset \subset \Omega_t(w_M )$ in $[t_0 - 4e, t_0 - 2e]$. By definition of $w_M$ and Lemma \ref{lem-cpt}, there exists sufficiently small $h \in (0,e)$ such that \textit{the constrained minimizing movements} $E_{t}^{h,M}$ starting from $\Omega_0(w_M )$ satisfies the following relations: $\Omega_{t}(\phi) \subset E_{t}^{h,M}$ in $[t_0 - 4e, t_0 - 2e]$ and
	\begin{align}
	\label{energy-bar4}
	|U^h_t| < \e, \left|\lambda[|E_{t}^{h,M}|] - \lambda[|\Omega_{t}(w_M)|]\right| < \e, \hbox{ and } d_H(U^h_t,U_t)<\e \hbox{ in } [t_0 - 4e,t_0]
	\end{align}
	where $U^h_t := \Omega_{t}(\tilde{\phi}) - E_{t}^{h,M}$. 
	
	\medskip
	
	Then, there exists $k \in \mathbb{N}$ such that $\Omega_{t_0 - hk}(\phi) \subset E_{t_0 - hk}^{h,M}$ and $U^h_{t_0 - h(k-1)}$ is nonempty.
	By $\Omega_{t}(\phi) \subset E_{t}^{h,M}$ in $[t_0 - 4e, t_0 - 2e]$, we have $t_1 := t_0 - h(k-1) \geq t_0 -2e$. Also, by \eqref{energy-bar4}, $U^h_{t_1} \subset \mathcal{N}_\e$ and thus \eqref{energy-bar5} holds in $U^h_{t_1}$.
	
	\medskip
	
	4. For simplicity let us denote sets 
\begin{align}
\label{prop_eqn_energy_2}
F_0 := E^{h,M}_{t_1-h}, F_h := E^{h,M}_{t_1}, \tilde{U}:= U^h_{t_1} \hbox{ and }\tilde{F_h} := E^{h,M}_{t_1} \cup \tilde{U}.
\end{align}
Let us show that $\tilde{F_h}\in A_{M}$. First, as $e \leq \frac{r_1 - r_0}{2}$, $\Omega_{t_1}(\tilde{\phi}) \in S_{r_0}$. Moreover, $E^h_{t_1}\in S_{r_0}$, and thus $\tilde{F_h} \in S_{r_0}$. On the other hand, since $\tilde{F_h} \subset F_h$,
\begin{align}
\label{step31}
d_H(\partial(\tilde{F_h} \cap F_0),\partial F_0) \leq d_H(\partial(F_h \cap F_0),\partial F_0)\leq Mh,
\end{align}
	
\medskip
	
Next, let us show that   $I_h(F_h;F_0) > I_h(\tilde{F_h},F_0)$.
Let us write out the difference of the energies:
\begin{align*}
I_h(F_h;F_0) - I_h(\tilde{F_h};F_0) &= \left(\mathrm{Per}(E_h) - \mathrm{Per}(\tilde{E_h})\right) + \left( - \Lambda[ |{F_h}|] +  \Lambda[ |\tilde{F_h}|]\right)  + \frac{1}{h}\left(\tilde{d}^2(F_h,F_0) - \tilde{d}^2(\tilde{F_h},F_0)\right).
\end{align*}
	Let us estimate the first term
	\begin{align*}
		\mathcal{I}_1:=\mathrm{Per}(F_h) - \mathrm{Per}(\tilde{F_h}) &\geq \int_{\partial{F_h}/\partial\tilde{F_h}} d\sigma - \int_{\partial\tilde{F_h}/\partial{F_h}} d\sigma
	\end{align*}
	Let $\vec{n}$ be the outward normal vector at each point of $\partial{F_h}/\partial\tilde{F_h}$ and $\partial\tilde{F_h}/\partial{F_h}$.  Note that, $-\frac{D\tilde{\phi}}{|D\tilde{\phi}|}(\cdot,t_1) \cdot \vec{n} \leq 1$ on $\partial{F_h}/\partial\tilde{F_h}$ and $-\frac{D\tilde{\phi}}{|D\tilde{\phi}|}(\cdot,t_1) \cdot \vec{n}  = 1 $ on $\partial\tilde{F_h}/\partial{F_h}$, and thus
\begin{align*}
\mathcal{I}_1 &\geq \int_{\partial{F_h}/\partial\tilde{F_h}} -\frac{D\tilde{\phi}}{|D\tilde{\phi}|}(x,t_1) \cdot \vec{n} d\sigma - \int_{\partial\tilde{F_h}/\partial{F_h}} -\frac{D\tilde{\phi}}{|D\tilde{\phi}|}(x,t_1) \cdot \vec{n} d\sigma = \int_{\partial \tilde{U}} \frac{D\tilde{\phi}}{|D\tilde{\phi}|}(x,t_1) \cdot \vec{n} d\sigma.
\end{align*}
Note that outward normal of $\tilde{U}$ is opposite to that of $\partial{F_h}/\partial\tilde{F_h}$. Finally, by divergence theorem, we conclude that
\begin{align}
\mathcal{I}_1 &\geq \int_{\tilde{U}} \nabla \cdot \frac{D\tilde{\phi}}{|D\tilde{\phi}|}(x,t_1) dx
\end{align}
	
	\medskip
	
	Next, since $\Lambda(\cdot)$ is $C^{1,1}$, we have 
	\begin{align}
	\label{eqn_energy_bar}
		\mathcal{I}_2:=- \Lambda[|{F_h}|] +  \Lambda[ |\tilde{F_h}|] &\geq   \lambda[|F_h|] |\tilde{U}| - C |\tilde{U}|^2
\hbox{ where } C := \sup_{|B_{r_0}| \leq z \leq  |B_R|} |\lambda'(z)| 
\end{align}
	
	Lastly we have
	\begin{align}
		\mathcal{I}_3:=\frac{1}{h}\tilde{d}^2(F_h,F_0) - \frac{1}{h}\tilde{d}^2(\tilde{F_h},F_0) &= -\frac{1}{h}\int_{U^h} \sd(x, F_0) dx
	\end{align}
	where $\sd(x, \Omega)$ is the signed distance function given in \eqref{eqn:sd}.
	Since $\Omega_{t_1-h}(\tilde{\phi}) \subset F_0$, it holds that $\sd(x, F_0) \leq \sd(x, \Omega_{t_1-h}(\tilde{\phi}))$ for all $x \in \R^n$. Moreover, since \eqref{energy-bar5} holds in $\tilde{U}$, we have
\begin{align}
\mathcal{I}_3 &\geq -\frac{1}{h}\int_{\tilde{U}} \sd(x, \Omega_{t_1-h}(\tilde{\phi})) dx \quad \geq -\int_{\tilde{U}}  \frac{\tilde{\phi}_t}{|D\tilde{\phi}|}(x,t_1) + \e dx, \label{eqn_energy_bar_2}
\end{align}
	
	\medskip
	
	Putting all terms together, we have
	\begin{align}
	\mathcal{I}_4:= I_h(F_h;F_0) - I_h(\tilde{F_h};F_0) &\geq  \int_{\tilde{U}} \left( \nabla \cdot  \frac{D\tilde{\phi}}{|D\tilde{\phi}|}(x,t_1) -  \frac{\tilde{\phi}_t}{|D\tilde{\phi}|}(x,t_1) + \lambda[|F_h|] \right) dx - \e|\tilde{U}| - C |\tilde{U}|^2,  \nonumber
	\end{align}
	Applying \eqref{energy-bar5} and \eqref{energy-bar4}, it holds that 
	\begin{align*}
	\mathcal{I}_4 &\geq  \int_{\tilde{U}} \left( \frac{\delta}{4} - \lambda[|\Omega_{t_1}(w_M)|] + \lambda[|F_h|] \right)dx    - \e|\tilde{U}| - C |\tilde{U}|^2 \geq |\tilde{U}| \left( \frac{\delta}{4} - 2\e - C |\tilde{U}| \right) >0
\end{align*}
	where the last inequality follows from the fact that $\e < \frac{\delta}{8+4C}$ and $|\tilde{U}| \leq \e$. 
	
	\medskip
	
	5. Lastly consider the case
	\begin{align*}
	\dfrac{\tilde{\phi}_t}{|D\tilde{\phi}|}(x_0,t_0) \leq  - M - \frac{\delta}{2}
	\end{align*}
	in the equation \eqref{energy-bar21}. 
	
	\	
	By parallel argument in step 2-4, for any $\e>0$, we can choose $e$ and $h$ sufficiently small in the definition of $\tilde{\phi}$ in \eqref{prop_eqn_energy_1} and the constrained minimizing movement $E_t^{h,M}$ such that for $F_0$ and $F_h$ as defined in step 4, $\Omega_{t_1 - h}(\tilde{\phi})$ is contained in $F_0$, $\Omega_{t_1}(\tilde{\phi}) \cap F_h$ is nonempty, and $\tilde{\phi}$ satisfies 
	\begin{align} \label{speed111}
	\dfrac{\tilde{\phi}_t}{|D\tilde{\phi}|}(x,t) \leq  - M - \frac{\delta}{2} \hbox{ and } \left| \dfrac{\tilde{\phi}_t}{|D\tilde{\phi}|}(x,t) - \dfrac{\tilde{\phi}_t}{|D\tilde{\phi}|}(x,t_0) \right| < \frac{\e}{2}
	\end{align}
	for $x \in \Omega_{t_1}(\tilde{\phi}) \setminus F_h$ and $t \in [t_1-h,t_1]$. Note that \eqref{step31} holds for $\tilde{F_h} := F_h \cup \Omega_{t_1}(\tilde{\phi})$.
	
	\medskip
	
	Let $x^*$ be a point in $(\partial \tilde{F_h}) \setminus (\partial F_h)$. As $\Omega_{t_1 - h}(\tilde{\phi})$ is contained in $F_0$ and $\Omega_t(\tilde{\phi})$ has a negative normal velocity, the point $x^* \in \Omega_{t_1}(\tilde{\phi}) \subset \Omega_{t_1-h}(\tilde{\phi}) \subset F_0$. Thus, $x^*$ is on $\partial(\tilde{F_h} \cap F_0)$, and
	\begin{align*}
		d_H(\partial(\tilde{F_h} \cap F_0),\partial F_0) &\geq \sup_{x \in \partial(\tilde{F_h} \cap F_0)} d(x,\partial F_0) \geq d(x^*,\partial F_0)
	\end{align*}
	Moreover, as $x^* \in \Omega_{t_1-h}(\tilde{\phi}) \subset F_0$, and \eqref{speed111}, we have
\begin{align}
d(x^*,\partial F_0) &\geq d(x^*,\partial \Omega_{t_1 - h}(\tilde{\phi})) > (M + \frac{\delta }{2})h - \frac{\e}{2} h > Mh, 
\end{align}
and this contradicts \eqref{step31}.
\end{proof} 

\section{Regularity}

In this section, we use notation from \cite{Huisken:69hmQ_XX} and \cite{Hui87}. Let $\partial \Omega_0$ be represented locally by some diffeomorphism, $F_0 : U \subset \R^{n-1} \rightarrow  F_0(U) \subset \partial\Omega_0$. Then, $\eqref{fixed}$ can be formulated into
\begin{align}
\begin{cases}
\frac{\partial}{\partial t} F(x,t) &= (\eta(t)-H(x,t)) \cdot \vec{n}(x,t), \quad \hbox{ for } x\in U, t\geq 0
\\  F(\cdot, 0) &= F_0
\end{cases}
\end{align}

The induced metric, its inverse matrix, and the second fundamental form are denoted by $\{g_{ij}\}$, $ \{g^{ij}\}$ and $A = \{h_{ij}\}$. Note that $g_{ij}$ and $h_{ij}$ can be computed as follows:
\begin{align}
g_{ij} = \left( \frac{\partial F}{\partial x_i}, \frac{\partial F}{\partial x_i} \right), \quad h_{ij} = -\left( \vec{n}, \frac{\partial^2 F}{\partial x_i \partial x_j} \right),
\end{align}
We use the following notion for the trace of the second fundamental from, $$H = g^{ij}h_{ij}, \quad |A|^2 = g^{ij} g^{kl} h_{ik} h_{jl}, \hbox{ and } C = g^{ij} g^{kl} g^{mn} h_{ik} h_{lm} h_{nj}.$$

The following lemma is parallel to Theorem 3.1 in \cite{EH} and Lemma 3.2 in \cite{Sheng:2010uw}.
\begin{lem}\label{regularity}
Let $u(x,t)$ be a solution of 
\begin{align}
\label{graphmcf2}
\frac{\partial u}{\partial t} = \sqrt{1+|Du|^2} \, \mathrm{ div} \left( \frac{Du}{ \sqrt{1+|Du|^2}} \right) + \eta(t)  \sqrt{1+|Du|^2}
\end{align}
in $Q_R=B_R(0) \times [0,R^2]$. Then for $0<t\leq R^2$, we have the interior gradient estimate
\begin{align}
\label{appen1}
|D^2u|^2(0,t) \leq K(1+ \sup_{Q_R} |Du|^6 )(\frac{1}{R^2} + \frac{1}{t})
\end{align}
where the constant $K = K( \|u\|_{L^\infty(Q_R)},\|\eta\|_{L^\infty([0,R^2])})$.
\end{lem}
\begin{proof}
First, by Corollary 1.2 in \cite{Hui87}, it holds that $(\frac{\partial}{\partial t} - \triangle)(|A|^2) = - 2 |\nabla A|^2 + 2|A|^4 - 2\eta C.$	

\medskip

Let us denote $v=\sqrt{1+|Du|^2}$. As Lemma 1.1 in \cite{EH} and Lemma 3.2 in \cite{Sheng:2010uw}, the function $v$ satisfies the equation
\begin{align}
v_t = \triangle v - |A|^2 v - \frac{2}{v} |\nabla v|^2.
\end{align}
Let us define $\phi(r) := \frac{r}{1-\delta r}$ and $g := |A|^2 \phi(v^2)$. Then, by the direct computation motivated from of Lemma 3.2 in \cite{Sheng:2010uw} and Theorem 3.1 in \cite{EH}, we have 
\begin{align*}
\mathcal{I}_1 :=\left(\frac{\partial}{\partial t} - \triangle \right) g &= ( - 2 |\nabla A|^2 + 2|A|^4- 2\eta C) \phi(v^2) + \left(- |A|^2 v - \frac{2}{v} |\nabla v|^2\right)  \times \frac{2v |A|^2}{(1-\delta v^2)^2}
\end{align*}
Note that
$ \delta^2 \phi(v^2) = \frac{1}{1-\delta v^2} - 1$,
 it holds that
\begin{align*}
\mathcal{I}_1  &=
-2\delta g^2 - 2 |\nabla A|^2 \phi(v^2) + \frac{-4 |A|^2 {|\nabla v|}^2 }{(1-\delta v^2)^2} -  2\eta C \phi(v^2),\\
&= - 2 \delta g^2  - 2 |\nabla A|^2 \phi(v^2) + \left( \frac{-2 \delta {|\nabla v|}^2 g }{(1-\delta v^2)}  + \frac{-2 |A|^2 {|\nabla v|}^2 }{( 1-\delta v^2)} \right) + \frac{-2 |A|^2 {|\nabla v|}^2 }{(1-\delta v^2)^2}-  2\eta C \phi(v^2).
\end{align*}
Now, choose $\delta := \frac{1}{2}\inf_{Q_R} v^{-2}$. Applying Young's inequality and $\nabla g = 2 A \nabla A \phi(v^2) + 2v |A|^2 \phi'(v^2) \nabla v$,
\begin{align*}
\phi  v^{-3}\langle \nabla g, \nabla v \rangle &\leq   |\nabla A|^2 \phi(v^2)  + \frac{|A|^2 |\nabla v|^2}{1-\delta v^2} + \frac{|A|^2 |\nabla v|^2}{(1-\delta v^2)^2}.
\end{align*}
Finally, from Young's inequality and $\phi(v^2) \geq v^2$, the last term of $\mathcal{I}_1$ is bounded by 
\begin{align}
|- 2\eta C \phi(v^2)| \leq 2K_1 g^{3/2} |v| \leq \delta g^2 + \frac{K_1^2 g v^2}{\delta}
\end{align}
for some constant $K_1:=K_1( \|\eta\|_{L^\infty([0,R^2])}))>0$.

\medskip

Putting all together, it holds that
$$\left( \frac{\partial}{\partial t} - \triangle\right) g \leq - 2 \delta g^2 +  \frac{-2 \delta {|\nabla v|}^2 g }{(1-\delta v^2)} - 2 \phi  v^{-3}\langle \nabla g, \nabla v \rangle +  \delta g^2 + \frac{K_1^2 g v^2}{\delta}.$$
The rest of proof is parallel to Theorem 3.1 in \cite{EH} and Lemma 3.2 in \cite{Sheng:2010uw}. 
Taking a cutoff function as in \cite{EH}, 
$\psi=\psi(r) = (R^2 - r)^2$ 
where $r=r(X,t)$ satisfies $r(X,0) \leq \frac{R^2}{2}$, 
\begin{align*}
\left| \left( \frac{\partial}{\partial t} - \triangle\right) r \right| \leq K_2 \hbox{ and } | \nabla r |^2 \leq K_2 r
\end{align*}
on $X = F(x,t)$ for some constant $K_2 = K_2( \|u\|_{L^\infty(Q_R)},\|\eta\|_{L^\infty([0,R^2])})>0$. 
It holds that
\begin{align*}
\left( \frac{\partial}{\partial t} - \triangle\right)[tg\psi] \leq -\delta g^2 \psi t - \vec{b} \cdot \nabla (tg\psi) + c\left( \left( 1+ \frac{1}{\delta v^2} \right) r + R^2 \right)tg + g\psi + \frac{K_1^2 g v^2}{\delta} \psi t
\end{align*}
where $\vec{b}= \vec{b}(v, \psi, \phi)$ and $c=c(K_2)$ is a constant (See equations (21) and (23) in \cite{EH} for details).

\medskip

Let $t_0$ be a maximizer of $m(T) := \sup\limits_{0 \leq t \leq T} \sup\limits_{ r(x,t) \leq R^2 } tg\psi$. Then, by parallel computation in Theorem 3.1 in \cite{EH}, we conclude that
\begin{align*}
\delta g^2 \psi t_0 \leq c\left( \left( 1+ \frac{1}{\delta v^2} \right) r + R^2 \right)t_0 g + g \psi + \frac{K_1^2 g v^2}{\delta} \psi t_0.
\end{align*}
Note that $\frac{R^4}{2}\leq \psi \leq R^4$ at $t=0$, $\phi(v^2) \geq v^2 \geq 1$, and $v^2 \leq \frac{1}{\delta}$. Thus, it holds that  
\begin{align}
|A|^2 \leq \frac{2}{\delta R^4}\left( cR^2\left( 2+ \frac{1}{\delta }\right) + \frac{R^4}{T} + \frac{K_1^2 R^4}{ \delta^2}    \right) \leq K\left( 1+ \frac{1}{\delta^3 } \right) \left( \frac{1}{T} + \frac{1}{R^2}  \right)
\end{align}
where $K = K(K_1, c)$, thus we conclude. 
\end{proof}

\section{Geometric properties}
\label{sec_star}

\begin{lem}\cite[Lemma 23 and 24]{Feldman:2014hb}
\label{lem-cpt}
Let us consider two sets $\Omega_1, \Omega_2 \in S_{r,R}$ for $R>r>0$. Then the following holds:
\begin{align}
\label{eqn:1cpt}
d_H(\Omega_1, \Omega_2) \leq d_H(\partial \Omega_1, \partial \Omega_2),\quad d_H(\partial \Omega_1, \partial \Omega_2) \lesssim_{r,R} d_H(\Omega_1, \Omega_2), \quad |\Omega_1 \Delta \Omega_2| \lesssim_{r,R} d_H(\Omega_1, \Omega_2), 
\end{align}
$$\Big| \tilde{d}(\Omega_1, E) - \tilde{d}(\Omega_2, E) \Big|, \Big| \tilde{d}(E,\Omega_1) - \tilde{d}(E,\Omega_2) \Big|  \lesssim_{r,R} d_H(\Omega_1, \Omega_2) \hbox{ for any } E \in S_{r,R},$$
\end{lem}

\begin{lem}
\label{lem-cpt2}
Let us consider two sets $\Omega_1, \Omega_2 \in S_{r,R}$ for $R>r>0$. Then the following holds:
\begin{align}
d_H(\Omega_1, \Omega_2)^{n+1} \leq w_n^{-1} \left( \frac{4R}{r} \right)^{n+1}    \tilde{d}^2(\Omega_1, \Omega_2) \hbox{ and } d_H(\Omega_1, \Omega_2)^{N+1} \leq w_n^{-1} \left( \frac{4R}{r} \right)^{n+1}    \tilde{d}^2(\Omega_2, \Omega_1)
\end{align}
\end{lem}

\begin{proof}
Due to the first inequality of \eqref{eqn:1cpt} in Lemma~\ref{lem-cpt}, it is enough to show that	
\begin{align*}
	d_H(\partial \Omega_1, \partial \Omega_2)^{n+1} \leq w_n^{-1} \left( \frac{4R}{r} \right)^{n+1} \tilde{d}^2(\Omega_1, \Omega_2) \hbox{ and }  d_H(\partial \Omega_1, \partial \Omega_2)^{n+1} \leq w_n^{-1} \left( \frac{4R}{r} \right)^{n+1} \tilde{d}^2(\Omega_2, \Omega_1).
	\end{align*}
	Without loss of generality, let us assume that $d_H(\partial \Omega_1, \partial \Omega_2) = \sup_{x \in \partial \Omega_1} d(x,\partial \Omega_2)$. Since $\partial \Omega_1$ and $\partial \Omega_2$ are compact, there exists $x_1 \in \partial \Omega_1$ and $x_2  \in \partial \Omega_2$ such that $ \sup_{x \in \partial \Omega_1} d(x,\partial \Omega_2) = d(x_1,\partial \Omega_2) =  |x_1-x_2|$. 
			Since $\Omega_2 \in S_r$, there exists $y \in \partial \Omega_2$ such that $x_1$ and $y$ are parallel. Note that we have $d(x_1,\partial \Omega_2)  \leq |x_1 - y|$.
	We argue for  the case $|x_1 | < |y|$. Since $x_1\in \partial\Omega_1$ and $y\in\partial\Omega_2$, there exists an exterior cone $EC(x_1,r)$ and an interior cone $IC(y,r)$ given in \eqref{eqn_ic} and \eqref{eqn_ec} such that $EC(x_1,r) \cap IC(y,r) \subset \Omega_2 \setminus \Omega_1$. Note that, for $\theta \in (0,\frac{\pi}{2})$ such that $\sin(\theta) = \frac{r}{R}$, we have	
$$(x_1+C(x_1,\theta)) \cap (y+C(-y,\theta)) \subset EC(x_1,r) \cap IC(y,r).$$
	
Note also that there is $\delta=\delta(r,R)$ such that 
\begin{align*}
B_{2\delta |x_1-y|}\left((x_1+y)/2\right) \subset (x_1+C(x_1,\theta)) \cap (y+C(-y,\theta)).
\end{align*}
Specifically, as $x_1$ and $y$ are parallel, the above inequality holds for
\begin{align}
\delta(r,R) = \frac{\sin (\theta)}{4} = \frac{r}{4R}.
\end{align}
Then, it holds that
\begin{align*}
\tilde{d}^2(\Omega_1,\Omega_2) &\geq \int_{\Omega_1 \triangle \Omega_2}{d(x,\partial \Omega_2)}dx \geq \int_{B_{\delta |x_1-y|}\left((x_1+y)/2\right)}{\delta |x_1-y|}dx= w_n \delta^{n+1} |x_1 - y|^{n+1}.
\end{align*}
The same inequality holds for $\tilde{d}^2(\Omega_2,\Omega_1)$ and thus we can conclude. Lastly, if $|x_1 | < |y|$, then we can apply the parallel arguments in
$(x_1+C(-x_1,\theta)) \cap (y+C(y,\theta)) \subset \Omega_1 \setminus \Omega_2.$
\end{proof}

\begin{lem}\cite[Lemma 24]{Feldman:2014hb}
	\label{beer}
	The metric space $(S_{r,R}, d_H)$ is compact:
	\begin{enumerate}
		\item
		Let us consider a sequence of sets $F_k \in S_{r,R}$ for $k \in \mathbb{N}$. Then $\{F_k\}_{k\in \mathbb{N}}$ has a subsequence that converges and any subsequential limit is also in $S_{r,R}$.
		\item
		Let $I$ be a compact interval in $[0, +\infty)$. Let us consider a sequence of evolving sets $F_k(\cdot) : I \rightarrow S_{r,R}$ for $k \in \mathbb{N}$. Assume that $\{F_k\}_{k\in \mathbb{N}}$ is equicontinuous in time, that is for all $\e>0$, there exists $\delta>0$ such that 
		$$d_H(F_k(t),F_k(s)) \leq \e $$
		for all $|t-s| \leq \delta$ and $k \in \mathbb{N}$.
		Then $\{F_k\}_{k\in \mathbb{N}}$ has a subsequence that converges uniformly on $I$ to an evolving set $F_\infty(\cdot) : I \rightarrow S_{r,R}$.
	\end{enumerate}
\end{lem}

\begin{lem}
\label{lem:ic}
For $r>0$ and $x \in \R^n$ such that $|x| \geq r$, it holds that 
\begin{align}
IC(x,r) = \left\{ \alpha x + (1- \alpha) y : \alpha \in (0,1), y \in B_r(0) \right\}.
\end{align}
Here, $IC(\cdot, \cdot)$ is given in \eqref{eqn_ic}.
\end{lem}

\begin{proof}
The proof is based on the geometry of interior cones describe in Figure~\ref{fig:ic}. Let us show that 
\begin{align}
\label{eqn:ic11}
\N:= \left\{ \alpha x + (1- \alpha) y : \alpha \in (0,1), \ \ y \in B_r(0) \right\} \subset IC(x,r).
\end{align}
For $z \in \N$, we fix $\alpha \in (0,1)$ and $y \in B_r(0)$ satisfying $z:= \alpha x + (1- \alpha) y$.
If $z \in B_r(0)$, then it can be checked that $z \in IC(x,r)$. Let us assume that $z \in B_r(0)^C$ and show that
\begin{align}
z \in (x+C(-x,\theta_{x,r})) \cap C\left(x,\frac{\pi}{2} - \theta_{x,r}\right).
\end{align}
Note that $x+C(-x,\theta_{x,r})$ is a convex set and $y \in B_r(0) \subset x+C(-x,\theta_{x,r})$ (See Figure~\ref{fig:ic}) and thus $z \in x+C(-x,\theta_{x,r})$. It remains to show that
\begin{align}
\label{eqn:ic14}
z \in C\left(x,\frac{\pi}{2} - \theta_{x,r}\right).
\end{align}
As $y \in B_r(0)$ and $z \in B_r(0)^C$, there  two intersection points $z_1$ and $z_2$ between $\partial B_r(0)$ and the line passing through $y$ and $z$ such that 
\begin{align*}
z_i :=  \alpha_i x + (1- \alpha_i) y \in \partial B_r(0) \hbox{ for } i=1,2 \hbox{ and } |x - z_1| < |x - z_2|
\end{align*}
for some $\alpha_1 \in (0, \alpha]$ and $\alpha_2 < 0$. As $z_1$ and $z_2$ are intersection points between a circle and a line, it holds that
\begin{align*}
|x|^2 - r^2 = |x - z_1| |x - z_2| \hbox{ and thus } |x - z_1| < \sqrt{ |x|^2 - r^2}.
\end{align*}
As $x \in C\left(x,\frac{\pi}{2} - \theta_{x,r}\right)$ and $d(x, \partial C\left(x,\frac{\pi}{2} - \theta_{x,r}\right)) = \sqrt{ |x|^2 - r^2}$ (See Figure~\ref{fig:ic}), we conclude that $z_1 \in C\left(x,\frac{\pi}{2} - \theta_{x,r}\right)$. As $C\left(x,\frac{\pi}{2} - \theta_{x,r}\right)$ is a convex set, we conclude \eqref{eqn:ic14} and thus \eqref{eqn:ic11} holds. 

\medskip

The opposite relation can be shown by similar geometric arguments. As $B_r(0) \subset \N$, it suffices to show that
\begin{align*}
z \subset \N \hbox{ for all } z \in \left\{ (x+C(-x,\theta_{x,r})) \cap C\left(x,\frac{\pi}{2} - \theta_{x,r}\right) \right\} \setminus B_r(0).
\end{align*}
Consider a line passing through $x$ and $z$, we can find a point $y \in B_r(0)$ such that $z = \alpha x + (1- \alpha) y$ for some $\alpha \in (0,1)$.
\end{proof}

\begin{lem}
\label{lem:iic}
For $x, z \in \R^n$ and $r>c>0$, assume that $|x| \geq r$ and $|z| < c$. Then, it holds that
\begin{align}
\label{eqn:1iic}
IC(x+z, r-c) \subset IC(x,r) + z.
\end{align}
Here, $IC(\cdot, \cdot)$ is given in \eqref{eqn_ic}.
\end{lem}

\begin{proof}
We claim that for $\alpha \in (0,1)$ and $y \in B_{r-c}(0)$, it holds that
\begin{align}
\label{eqn:iic11}
\alpha (x+z) + (1-\alpha) y \in IC(x,r) + z.
\end{align}
Note that
\begin{align*}
\alpha (x+z) + (1-\alpha) y - z = \alpha x + (1-\alpha) (y-z).
\end{align*}
As $y \in B_{r-c}(0)$ and $z \in B_{c}(0)$, we have $y-z \in B_{r}(0)$. From Lemma~\ref{lem:ic}, we have \eqref{eqn:iic11}. From Lemma~\ref{lem:ic} again, we conclude \eqref{eqn:1iic}.
\end{proof}

\begin{lem}
\label{lem:haus}
Let us consider two sets $\Omega_1, \Omega_2 \in S_{r,R}$ for $R>r>0$. Then the following holds:
\begin{align}
\label{eqn:1haus}
\sup_{x \in \partial \Omega_2} d(x, \partial \Omega_1) \leq \frac{R}{r} \sup_{x \in \partial \Omega_1} d(x, \partial \Omega_2).
\end{align}
\end{lem}

\begin{proof}
If $\sup_{x \in \partial \Omega_2} d(x, \partial \Omega_1) = 0$, then \eqref{eqn:1haus} holds. We suppose that $\sup_{x \in \partial \Omega_2} d(x, \partial \Omega_1)> 0$. As $\Omega_2 \in S_{r,R}$, there exists $x_2 \in \partial \Omega_2$ such that
\begin{align}
\sup_{x \in \partial \Omega_2} d(x, \partial \Omega_1) = d(x_2, \partial \Omega_1) =:l>0.
\end{align}
As a consequence, we have 
\begin{align}
\label{eqn:haus13}
B_l(x_2) \subset \Omega_1^C \hbox{ and } B_l(x_2) \subset \Omega_1. 
\end{align}
Let us assume the former one. As $\Omega_1 \in S_{r,R}$, there exists $x_1 \in \partial \Omega_1$ such that $x_1$ is in the line segment between the origin and $x_2$. From \eqref{eqn:haus13}, $|x_1 - x_2| \geq l$. From the interior cone property of $S_{r,R}$ in Lemma~\ref{star}, it holds that
\begin{align}
d(x_1, \partial \Omega_2) \geq d(x_1, \partial IC(x_2, r)) \geq \frac{lr}{R}
\end{align}
and we conclude \eqref{eqn:1haus}. The latter case in \eqref{eqn:haus13} can be shown by the parallel arguments.
\end{proof}

\bibliographystyle{abbrv}
\bibliography{Paper_MCF_Final}

\begin{thebibliography}{10}

\bibitem{Ale58}
A.~D. Alexandrov.
\newblock {Uniqueness theorems for surfaces in the large V}.
\newblock {\em Vestnik, Leningrad University}, 13(19):5--8, 1958.

\bibitem{Ale62}
A.~D. Alexandrov.
\newblock A characteristic property of spheres.
\newblock {\em Annali di Matematica Pura ed Applicata}, 58(1):303--315, 1962.

\bibitem{Almgren:1993gw}
F.~Almgren, J.~E. Taylor, and L.~Wang.
\newblock {Curvature-driven flows: a variational approach}.
\newblock {\em SIAM Journal on Control and Optimization}, 31(2):387--438, 1993.

\bibitem{AAG1995}
S.~Altschuler, S.~B. Angenent, and Y.~Giga.
\newblock Mean curvature flow through singularities for surfaces of rotation.
\newblock {\em The Journal of Geometric Analysis}, 5(3):293--358, 1995.

\bibitem{Anonymous:5CWDAe2w}
B.~Andrews.
\newblock {Volume-preserving anisotropic mean curvature flow}.
\newblock {\em Indiana University mathematics journal}, 50(2):783--827, 2001.

\bibitem{A1}
S.~Angenent.
\newblock {Parabolic equations for curves on surfaces: part I. Curves with
  p-integrable curvature}.
\newblock {\em Annals of Mathematics}, 132(3):451--483, 1990.

\bibitem{A2}
S.~Angenent.
\newblock {Parabolic equations for curves on surfaces: part II. Intersections,
  blow-up and generalized solutions}.
\newblock {\em Annals of Mathematics}, 133(1):171--215, 1991.

\bibitem{Barles:2013gx}
G.~Barles.
\newblock {An introduction to the theory of viscosity solutions for first-order
  Hamilton-Jacobi equations and applications}.
\newblock In {\em Hamilton-Jacobi equations: approximations, numerical analysis
  and applications}, pages 49--109. Springer, Heidelberg, Berlin, Heidelberg,
  2013.

\bibitem{Barles:1993gaba}
G.~Barles, H.~M. Soner, and P.~E. Souganidis.
\newblock {Front Propagation and Phase Field Theory}.
\newblock {\em SIAM Journal on Control and Optimization}, 31(2):439--469, Mar.
  1993.

\bibitem{BCCN09}
G.~Bellettini, V.~Caselles, A.~Chambolle, and M.~Novaga.
\newblock {The volume preserving crystalline mean curvature flow of convex sets
  in}.
\newblock {\em Journal de Math{\'e}matiques Pures et Appliqu{\'e}es},
  92(5):499--527, 2009.

\bibitem{Brakke}
K.~A. Brakke.
\newblock {\em The Motion of a Surface by Its Mean Curvature.(MN-20)}.
\newblock Princeton University Press, 2015.

\bibitem{Caffarelli:2005tk}
L.~Caffarelli and S.~Salsa.
\newblock {\em {A geometric approach to free boundary problems}}, volume~68 of
  {\em Graduate Studies in Mathematics}.
\newblock American Mathematical Society, 2005.

\bibitem{Chambolle:2004ds}
A.~Chambolle.
\newblock {An algorithm for mean curvature motion}.
\newblock {\em Interfaces and Free Boundaries. Mathematical Modelling, Analysis
  and Computation}, 6(2):195--218, 2004.

\bibitem{CMNP19}
A.~Chambolle, M.~Morini, M.~Novaga, and M.~Ponsiglione.
\newblock Existence and uniqueness for anisotropic and crystalline mean
  curvature flows.
\newblock {\em Journal of the American Mathematical Society}, 2019.

\bibitem{CMP15}
A.~Chambolle, M.~Morini, and M.~Ponsiglione.
\newblock Nonlocal curvature flows.
\newblock {\em Archive for Rational Mechanics and Analysis}, 218(3):1263--1329,
  2015.

\bibitem{Chen:1991u}
Y.~G. Chen, Y.~Giga, and S.~Goto.
\newblock {Uniqueness and existence of viscosity solutions of generalized mean
  curvature flow equations}.
\newblock {\em Journal of Differential Geometry}, 1991.

\bibitem{Crandall:1992kn}
M.~G. Crandall, H.~Ishii, and P.-L. Lions.
\newblock {User{\textquoteright}s guide to viscosity solutions of second order
  partial differential equations}.
\newblock {\em Bulletin of the American Mathematical Society}, 27(1):1--67,
  1992.

\bibitem{EH}
K.~Ecker and G.~Huisken.
\newblock {Interior estimates for hypersurfaces moving by mean curvature}.
\newblock {\em Inventiones mathematicae}, 105(3):547--569, 1991.

\bibitem{Escher:1998}
J.~Escher and G.~Simonett.
\newblock {The volume preserving mean curvature flow near spheres}.
\newblock {\em Proceedings of the American Mathematical Society},
  126(9):2789--2796, 1998.

\bibitem{Evans1992}
L.~C. Evans, H.~M. Soner, and P.~E. Souganidis.
\newblock {Phase transitions and generalized motion by mean curvature}.
\newblock {\em Communications on Pure and Applied Mathematics},
  45(9):1097--1123, Oct. 1992.

\bibitem{Evans:1999tb}
L.~C. Evans and J.~Spruck.
\newblock {Motion of level sets by mean curvature. I}.
\newblock {\em Journal of Differential Geometry}, 33(3):635--681, 1991.

\bibitem{Anonymous:JTDBs7kP}
L.~C. Evans and J.~Spruck.
\newblock {Motion of level sets by mean curvature. II}.
\newblock {\em Transactions of the American Mathematical Society},
  330(1):321--332, 1992.

\bibitem{Anonymous:UpA3UBgR}
L.~C. Evans and J.~Spruck.
\newblock {Motion of level sets by mean curvature III}.
\newblock {\em The Journal of Geometric Analysis}, 2(2):121--150, 1992.

\bibitem{Evans:1995iw}
L.~C. Evans and J.~Spruck.
\newblock {Motion of level sets by mean curvature IV}.
\newblock {\em The Journal of Geometric Analysis}, 5(1):77--114, 1995.

\bibitem{Feldman:2014hb}
W.~M. Feldman and I.~C. Kim.
\newblock {Dynamic stability of equilibrium capillary drops}.
\newblock {\em Archive for Rational Mechanics and Analysis}, 211(3):819--878,
  2014.

\bibitem{Gig06}
Y.~Giga.
\newblock {\em {Surface evolution equations: a level set approach}}, volume~99
  of {\em Monographs in Mathematics}.
\newblock Birkh{\"a}user Basel, Basel, 2006.

\bibitem{GGIS}
Y.~Giga, S.~Goto, H.~Ishii, and M.-H. Sato.
\newblock Comparison principle and convexity preserving properties for singular
  degenerate parabolic equations on unbounded domains.
\newblock {\em Indiana University Mathematics Journal}, pages 443--470, 1991.

\bibitem{Grunewald:2011cb}
N.~Grunewald and I.~Kim.
\newblock {A variational approach to a quasi-static droplet model}.
\newblock {\em Calculus of Variations and Partial Differential Equations},
  41(1-2):1--19, 2011.

\bibitem{Huisken:69hmQ_XX}
G.~Huisken.
\newblock {Flow by mean curvature of convex surfaces into spheres}.
\newblock {\em Journal of Differential Geometry}, 20(1):237--266, 1984.

\bibitem{Hui87}
G.~Huisken.
\newblock {The volume preserving mean curvature flow}.
\newblock {\em {Journal f{\"u}r die reine und angewandte Mathematik (Crelles
  Journal)}}, 1987(382):35--48, 1987.

\bibitem{HS1}
G.~Huisken and C.~Sinestrari.
\newblock {Convexity estimates for mean curvature flow and singularities of
  mean convex surfaces}.
\newblock {\em Acta Mathematica}, 183(1):45--70, 1999.

\bibitem{HS2}
G.~Huisken and C.~Sinestrari.
\newblock {Mean curvature flow singularities for mean convex surfaces}.
\newblock {\em Calculus of Variations and Partial Differential Equations},
  8(1):1--14, 1999.

\bibitem{ImbSil13}
C.~Imbert and L.~Silvestre.
\newblock An introduction to fully nonlinear parabolic equations.
\newblock In {\em An introduction to the K{\"a}hler-Ricci flow}, pages 7--88.
  Springer, 2013.

\bibitem{KK19}
I.~Kim and D.~Kwon.
\newblock Volume preserving mean curvature flow for star-shaped sets.
\newblock {\em arXiv preprint arXiv:1808.04922 [math.AP]}, 2018.
\newblock \url{https://arxiv.org/abs/1808.04922}.

\bibitem{Kim:2002ta}
I.~C. Kim.
\newblock Uniqueness and existence results on the hele-shaw and the stefan
  problems.
\newblock {\em Archive for Rational Mechanics and Analysis}, 168(4):299--328,
  2003.

\bibitem{Kim:2005fi}
I.~C. Kim.
\newblock {A free boundary problem with curvature}.
\newblock {\em Communications in Partial Difference Equations},
  30(1-2):121--138, 2005.

\bibitem{Lin}
L.~Lin.
\newblock {Mean curvature flow of star-shaped hypersurfaces}.
\newblock {\em arXiv:1508.01225v1 [math.DG]}, 2015.
\newblock \url{https://arxiv.org/abs/1508.01225}.

\bibitem{LS95}
S.~Luckhaus and T.~Sturzenhecker.
\newblock {Implicit time discretization for the mean curvature flow equation}.
\newblock {\em Calculus of Variations and Partial Differential Equations},
  3(2):253--271, 1995.

\bibitem{Maggi:2012tg}
F.~Maggi.
\newblock {\em {Sets of Finite Perimeter and Geometric Variational Problems}}.
\newblock An Introduction to Geometric Measure Theory. Cambridge University
  Press, 2012.

\bibitem{Sheng:2010uw}
W.~Sheng and X.-J. Wang.
\newblock {Regularity and singularity in the mean curvature flow}.
\newblock In {\em Trends in partial differential equations}, volume~10 of {\em
  Advanced Lectures in Mathematics}, pages 399--436. Higher Education Press and
  International Press, 2010.

\bibitem{Smo1}
K.~Smoczyk.
\newblock {Starshaped hypersurfaces and the mean curvature flow}.
\newblock {\em Manuscripta Mathematica}, 95(2):225--236, 1998.

\bibitem{soner1993}
H.~M. Soner and P.~E. Souganidis.
\newblock Singularities and uniqueness of cylindrically symmetric surfaces
  moving by mean curvature.
\newblock {\em Communications in partial differential equations},
  18(5-6):859--894, 1993.

\end{thebibliography}

\end{document}